\documentclass{amsart}
\title{A Permutation Module Deligne Category and Stable Patterns of Kronecker Coefficients}
\author{Christopher Ryba}
\usepackage{hyperref}
\usepackage{fullpage}
\usepackage{amsmath}
\usepackage{amsfonts}
\usepackage{amssymb}
\usepackage{comment}
\usepackage{ytableau}
\usepackage{subfig}
\usepackage{todonotes}
\usepackage{graphicx, caption}

\input tikz.tex
\usetikzlibrary{cd}

\DeclareMathOperator{\Hom}{Hom}
\DeclareMathOperator{\End}{End}
\DeclareMathOperator{\Ind}{Ind}

\DeclareMathOperator{\Sym}{Sym}
\DeclareMathOperator{\Mat}{Mat}
\DeclareMathOperator{\Id}{Id}

\begin{document}
\maketitle

\newtheorem{theorem}{Theorem}[section]
\newtheorem{lemma}[theorem]{Lemma}
\newtheorem{proposition}[theorem]{Proposition}
\newtheorem{corollary}[theorem]{Corollary}
\newtheorem{example}[theorem]{Example}
\newtheorem{remark}[theorem]{Remark}
\newtheorem{definition}[theorem]{Definition}

\begin{abstract}
Deligne's category $\underline{{\rm Rep}}(S_t)$ is a tensor category depending on a parameter $t$ ``interpolating'' the categories of representations of the symmetric groups $S_n$. We construct a family of categories $\mathcal{C}_\lambda$ (depending on a vector of variables $\lambda = (\lambda_1, \lambda_2, \ldots, \lambda_l)$, that may be specialised to values in the ground ring) which are module categories over $\underline{{\rm Rep}}(S_t)$. The categories $\mathcal{C}_\lambda$ are defined over any ring and are constructed by interpolating permutation representations. Further, they admit specialisation functors to $S_n$-mod which are tensor-compatible with the functors $\underline{{\rm Rep}}(S_t) \to S_n$-mod. We show that $\mathcal{C}_\lambda$ can be presented using the Kostant integral form of Lusztig's universal enveloping algebra $\dot{U}(\mathfrak{gl_{\infty}})$, and exhibit a categorification of some stability properties of Kronecker coefficients.
\end{abstract}

\tableofcontents

\section{Introduction}
\noindent
In this paper, we consider Deligne's category $\underline{{\rm Rep}}(S_t)$, a tensor category depending ``polynomially'' on a parameter $t$. It is usually defined over a field of characteristic zero (although it can be defined over any commutative ring) and may be thought of as an interpolation of the representation categories of the symmetric groups $S_n$ as tensor categories. There are ``specialisation'' functors $\underline{{\rm Rep}}(S_t) \to S_n$-mod which are full and essentially surjective in characteristic zero. Harman (\cite{NateStability}, see also chapter 4 of \cite{NateThesis}) introduced a different construction, $\underline{{\rm Perm}}_t$, based on permutation modules which he used to prove properties of decomposition numbers for symmetric groups. Using work of Comes and Ostrik \cite{OC}, we show that over a field of characteristic zero, $\underline{{\rm Rep}}(S_t)$ and $\underline{{\rm Perm}}_t$ are equivalent as tensor categories. Thus the two categories may be thought of as two different ``integral forms'' of the same category.
\newline \newline \noindent
We also construct a family of categories $\mathcal{C}_\lambda$ (depending on $\lambda = (\lambda_1, \lambda_2, \ldots, \lambda_l)$) also by interpolating permutation modules. The $\mathcal{C}_\lambda$ are themselves module categories over $\underline{{\rm Perm}}_t$, and the case $l=1$ yields $\underline{{\rm Perm}}_{\lambda_1}$. The $\mathcal{C}_\lambda$ admit specialisation functors to $S_n$-mod of their own which are compatible with those on $\underline{{\rm Perm}}_t$ (the functor $\underline{{\rm Perm}}_t \otimes \mathcal{C}_\lambda \to \mathcal{C}_\lambda$ specialises to the usual tensor product $S_n\mbox{-mod} \otimes S_n\mbox{-mod} \to S_n\mbox{-mod}$). The categories $\mathcal{C}_\lambda$ and their specialisation functors are defined over any commutative ring. However, although the Deligne category $\underline{{\rm Rep}}(S_t)$ is a semisimple abelian category when $t \notin \mathbb{Z}_{\geq 0}$, $\mathcal{C}_\lambda$ is only Karoubian (and is not Krull-Schmidt in general).
\newline \newline \noindent
The categories $\mathcal{C}_\lambda$ can be presented using the Kostant integral form of Lusztig's universal eneveloping algebra $\dot{U}(\mathfrak{gl}_\infty)$; in particular, the $\hom$-spaces are given by weight spaces in this algebra. This is readily leveraged to categorify the $(|\lambda|, \lambda, \lambda)$ Kronecker coefficient stability pattern described by Stembridge in \cite{Stembridge}, which we now describe.
\newline \newline \noindent
Suppose that $\alpha, \beta, \gamma, \lambda, \mu, \nu$ are partitions such that $|\alpha|=|\beta|=|\gamma|$ and $|\lambda|=|\mu|=|\nu|$. For $n \in \mathbb{Z}_{\geq 0}$, one may consider the Kronecker coefficients (tensor product multiplicities for symmetric groups)
\[
k_{\alpha + n\lambda, \beta + n \mu}^{\gamma + n \nu},
\]
where addition and scalar multiplication are defined componentwise. Stembridge conjectured that for any fixed triple $(\lambda, \mu, \nu)$, if there exist $\alpha, \beta, \gamma$ such that the limit of the above sequence exists as $n \to \infty$, then the sequence has a limit for any choice of $\alpha, \beta, \gamma$. This conjecture was proved by Sam and Snowden in \cite{SS}, using geometric methods. It remains an open problem to describe which triples $(\lambda, \mu, \nu)$ exhibit this stability property. Stembridge observes that all triples of the form $(|\lambda|, \lambda, \lambda)$ satisfy this condition (see Example 6.3 $(a)$ of \cite{Stembridge}). Our categorification illustrates this for all $\lambda$ simultaneously.
\newline \newline \noindent
The outline of the paper is as follows. In Section 2, we cover the necessary background on partition combinatorics, Schur algebras, and $\underline{{\rm Rep}}(S_t)$. Then, in Section 3, we construct the categories $\mathcal{C}_\lambda$ and discuss their relation to Deligne's category as well as their specialisation functors and tensor structure. In Section 4, we show how all the previous structure can be expressed in a more Lie-theoretic way using a presentation of Schur algebras as a quotient of $\dot{U}(\mathfrak{gl}_n)$ due to Doty and Giaquinto \cite{DG}. We conclude with Section 5 where we prove a categorified version of a stability property of Kronecker coefficients. The appendix sketches the proof that a certain example of $\mathcal{C}_\lambda$ fails to be Krull-Schmidt.

\subsection{Acknowledgements}
\noindent
The idea for this paper (categorification of Stembridge's stability patterns for Kronecker coefficients) was suggested by Pavel Etingof, following conversations with Greta Panova. The author would like to thank Pavel Etingof for useful conversations and Andrew Mathas for directing the author to \cite{DJ}.

\section{Preliminaries}
\subsection{Partitions and Compositions and Young Subgroups}
Recall that a finite weakly decreasing sequence of nonnegative integers is called a partition. We consider partitions that differ only by trailing zeroes to be equivalent. Suppose that $\lambda = (\lambda_1, \lambda_2, \ldots, \lambda_l)$ is a partition. The $\lambda_i$ are called the parts of $\lambda$. The length of $\lambda$, denoted $l(\lambda)$, is the number of nonzero parts of $\lambda$. The size of $\lambda$, denoted $|\lambda|$, is the sum of the parts of $\lambda$. The notation $\lambda \vdash n$ means that $\lambda$ is a partition of size $n$. It is common to use the notation $\lambda = (1^{m_1} 2^{m_2} \cdots )$ to mean that $\lambda$ has $m_i$ parts equal to $i$.
\newline \newline \noindent
There is a partial order on partitions of a fixed size called the dominance ordering. We write $\mu \unrhd \nu$ if 
\[
\mu_1 + \mu_2 + \cdots + \mu_i \geq \nu_1 + \nu_2 + \cdots + \nu_i
\]
for all $i \geq 1$, with equality holding for $i$ sufficiently large.
\newline \newline \noindent
A composition is a finite sequence of nonnegative integers. The length and size of compositions is defined identically to partitions. Two compositions that differ by trailing zeroes are considered equivalent (but not intermediate zeroes). To indicate that $\alpha$ is a composition of $n$, we write $\alpha \models n$. We write $\Lambda(n,d)$ for the set of compositions of $d$ into at most $n$ parts. The concatenation of sequences $\lambda, \mu$ (usually partitions or compositions) is denoted $(\lambda, \mu)$.
\newline \newline \noindent
A set partition $\alpha = \{ \alpha_i \}_{i \in I}$ of a set $U$ is a family of subsets $\alpha_i \subseteq U$ such that $\amalg_{i \in I} \alpha_i = U$. We will only consider finite sets $U$ and finite index sets $I$. We say that a set partition $\alpha$ is a coarsening of a set partition $\beta$ of the same set, if each part of $\alpha$ is a union of parts of $\beta$.
\newline \newline \noindent
For a composition $\alpha $, we define a Young subgroup, $S_\alpha$, of the symmetric group $S_{|\alpha|}$ as follows:
\[
S_\alpha = \prod_{i=1}^{l(\alpha)} S_{\alpha_i}.
\]
We refer to the $S_{\alpha_i}$ as factor groups of $S_\alpha$. Note that $S_\alpha$ embeds in $S_{|\alpha|}$ in the obvious way. For example, if $\alpha = \{1,3,2\}$, we have 
\[
S_\alpha = \Sym(\{1\})\times \Sym(\{2,3,4\}) \times \Sym(\{5,6\}) \subseteq \Sym(\{1,2,3,4,5,6\}) = S_6.
\]
A Young subgroup $S_\alpha$ is conjugate to any Young subgroup $S_\beta$ where $\beta$ is obtained from $\alpha$ by reordering parts. In particular, each $S_\alpha$ is conjugate to a unique $S_\lambda$ where $\lambda$ is a partition (obtained by sorting the parts of $\alpha$ in decreasing order).
\newline \newline
\noindent
We will need to understand the  $(S_\alpha, S_\beta)$-double cosets in $S_n$, where $\alpha, \beta \models n$. The following result is well known.
\begin{proposition} \label{young_double_cosets}
For two compositions $\alpha, \beta \models n$, the $S_\alpha \backslash S_n / S_\beta$ double cosets are indexed by $l(\beta) \times l(\alpha)$ matrices $A = (a_{ij})$ with entries in $\mathbb{Z}_{\geq 0}$ satisfying the following:
\begin{eqnarray*}
\sum_{i=1}^{l(\beta)} a_{ij} &=& \alpha_j, \\
\sum_{j=1}^{l(\alpha)} a_{ij} &=& \beta_i.
\end{eqnarray*}
A permutation $\sigma \in S_n$ belongs to the double coset indexed by $A$ if and only if the number of elements of $\{1,2,\ldots, n\}$ permuted by the factor group $S_{\beta_i}$ mapped by $\sigma$ to elements permuted by the factor group $S_{\alpha_j}$ is $a_{ij}$.
\end{proposition}
\noindent
We will typically use the same notation for a double coset representative and the associated matrix (e.g. $q$ and $q_{ij}$), and we usually use the letters $p,q,r,s$ for double coset representatives. It will be convenient for us to depict double cosets using the corresponding matrices, for example, the following matrix may be interpreted as a $(S_{(4,2)}, S_{(1,3,2)})$-double coset in $S_6$:
\[
\left( \begin{array}{cc}
1 & 0 \\
2 & 1 \\
1 & 1
\end{array} \right).
\]
\begin{remark} \label{reflection_remark}
If $q$ is a $(S_\alpha, S_\beta)$-double coset representative, then $q^{-1}$ is a $(S_\beta, S_\alpha)$-double coset representative. This map is well defined on the level of double cosets, and the matrix corresponding to $S_\beta q^{-1} S_\alpha$ is the transpose of the matrix of $S_\alpha q S_\beta$.
\end{remark}
\noindent
Finally, we recall two families of representations of the symmetric group $S_n$. Given $\alpha \models n$, we write $M^\alpha$ for the permutation module induced from the trivial $S_\alpha$-module, and for $\lambda \vdash n$, we write $S^{\lambda}$ for the Specht module labelled by the partition $\lambda$. Over $\mathbb{Q}$ the Specht modules are irreducible representations, and for partitions $\mu$ and $\lambda$, $\dim_{\mathbb{Q}}(\Hom_{S_n}(S^\lambda, M^\mu))$ is equal to the Kostka number $K_{\lambda, \mu}$. Note that $K_{\lambda, \mu}$ is zero unless $\lambda \unrhd \mu$, and equal to 1 if $\mu = \lambda$.
\begin{lemma} \label{sss_lemma}
Let us work with modules over $\mathbb{Q}S_n$. Consider $\Hom(M^\mu, S^\mu)$ as a right module for $\End(M^\mu)$ and $\Hom(S^\lambda, M^\lambda)$ as a left module for $\End(M^\lambda)$. If $N$ is any $S_n$-module,
\[
\Hom(S^\lambda, N \otimes S^{\mu}) = 
\Hom(M^\mu, S^\mu) \otimes_{\End_{\mathbb{Q}S_n}(M^\mu)} \Hom(M^\lambda, N \otimes M^\mu) \otimes_{\End_{\mathbb{Q}S_n}(M^\lambda)}\Hom(S^\lambda, M^\lambda),
\]
and this is functorial in $N$.
\end{lemma}
\begin{proof}
Because $S^\lambda$ occurs with multiplicity one in $M^\lambda$, any map $S^\lambda \to N \otimes M^\mu$ factors through $M^\lambda$ by semisimplicity of representations of $S_n$ over $\mathbb{Q}$, and the map $S^\lambda \to M^\lambda$ is unique up to scalar multiplication. An analogous property holds for maps into $S^\mu$.
\end{proof}

\subsection{Schur Algebras}
We recall some properties of Schur algebras that will be useful. Unless indicated otherwise, we work over an arbitrary commutative ring $R$. We let $V = R^{\oplus n}$ have standard basis $v_1, v_2, \ldots, v_n$.

\begin{definition}
Note that $V^{\otimes d}$ carries an action of the symmetric group $S_d$ by permutation of tensor factors. The Schur algebra, $S(n, d)$, is defined to be the commutant $\End_{S_d}(V^{\otimes d})$. When we wish to make the ground ring $R$ explicit, we will write $S_R(n,d)$.
\end{definition}

\begin{lemma}
Let us write $(V^{\otimes d})_\alpha$ for the $\alpha$-weight space of $V^{\otimes d}$ (i.e. the span of $v_{i_1} \otimes v_{i_2} \otimes \cdots \otimes v_{i_d}$ where the number of $i_j$ equal to $k$ is $\alpha_k$). Then $(V^{\otimes d})_\alpha$ is preserved by the $S_d$ action, and is isomorphic to $M^\alpha$. Over $\mathbb{Q}$, for a partition $\lambda$ of $d$, the image of $S^\lambda$ inside $(V^{\otimes d})_\lambda = M^\lambda$ consists of highest-weight vectors for $GL(V)$ of weight $\lambda$.
\end{lemma}
\begin{proof}
It is clear that $(V^{\otimes d})_\alpha$ is closed under the $S_d$ action, because there was no restriction on the ordering of the $i_j$. Moreover, the relevant set of multi-indices $(i_1, i_2, \ldots, i_d)$ has a transitive action of $S_d$ and the stabiliser of the unique sorted multi-index is precisely $S_\alpha$. This means the corresponding permutation representation is $M^\alpha$.
\newline \newline \noindent
Let us consider the action of $\Id + E_{ij} \in GL(V)$ on $M^\lambda$, where $E_{ij}$ is an elementary matrix with $i > j$. The effect of $E_{ij}$ is to replace $v_j$ with $v_i$, which takes an element in $M^\lambda$ to $M^\mu$, where $\mu$ is obtained from $\lambda$ by decrementing $\lambda_j$ and incrementing $\lambda_i$. Because $\mu \unrhd \lambda$ and $\mu \neq \lambda$, the Kostka number $K_{\lambda, \mu}$ is zero. Hence, there is no copy of $S^\lambda$ in $M^\mu$, and so this means that $\Id + E_{ij}$ must fix $S^\lambda$ (viewed as a subset of $M^\lambda$) pointwise. 
\end{proof}

\begin{proposition} \label{schur_basis_theorem}
There is a basis $\xi_q$ of $\Hom_{S_d}(M^\alpha, M^\beta)$ indexed by double cosets $q \in S_{\alpha} \backslash S_r / S_{\beta}$, where $\alpha, \beta \models d$. Moreover, if $M^\alpha$ is viewed as the permutation representation on cosets of $S_\alpha$, and similarly for $M^\beta$, $\xi_q$ acts on $gS_\alpha$ by sending it to $gS_\alpha q S_\beta$, where this expression is to be interpreted as the sum of its constituent $S_\beta$ cosets.
\end{proposition}

\begin{proof}
We apply the Mackey formula for induced representations. Below, $R$ indicates a trivial representation. Note that $M^* = \Hom_R(M, R)$, and $R^* = R$.
\begin{eqnarray*}
\Hom_{S_d}(M^\alpha, M^\beta) &=& ((M^\alpha)^* \otimes M^\beta )^{S_d} \\
&=& (\Ind_{S_\alpha}^{S_d}(R) \otimes \Ind_{S_\beta}^{S_d}(R))^{S_d} \\
&=& \left( \bigoplus_{q \in S_\alpha \backslash S_d / S_\beta} \Ind_{S_\alpha \cap q S_\beta q^{-1}}^{S_d}(R \otimes qR)\right)^{S_d} \\
&=&  \bigoplus_{q \in S_\alpha \backslash S_d / S_\beta} \Ind_{S_\alpha \cap q S_\beta q^{-1}}^{S_d}(R)^{S_d} \\
&=& \bigoplus_{q \in S_\alpha \backslash S_d / S_\beta} R.
\end{eqnarray*}
We let $\xi_q$ be the generator of the summand associated to the double-coset defined by $q$. The claim about the action follows by retracing the construction of $\xi_q$.
\end{proof}

\begin{remark}
An identical calculation shows that
\[
M^\alpha \otimes M^\beta  = \bigoplus_{q \in S_\alpha \backslash S_d / S_\beta} M^{q},
\]
where $q$ is viewed as a composition whose parts are the entries of the matrix associated to the double-coset defined by $q$. It will be convenient to use the notation $M^q$ when considering tensor products.
\end{remark}

\begin{remark} \label{associator_remark}
To understand the associator corresponding to this tensor product, we consider $M^\alpha \otimes M^\beta \otimes M^\gamma$. Regardless of how the tensor product is parenthesised, we obtain the sum of $M^\theta$, where $\theta = (\theta_{ijk})$ is a $3$-tensor satisfying the following conditions:
\begin{eqnarray*}
\sum_{ij} \theta_{ijk} &=& \gamma_k \\
\sum_{ik} \theta_{ijk} &=& \beta_j \\
\sum_{jk} \theta_{ijk} &=& \alpha_i.
\end{eqnarray*}
When we consider $(M^\alpha \otimes M^\beta) \otimes M^\gamma$, $\theta_{ijk}$ is viewed as a matrix whose rows are indexed by pairs $(i,j)$ and whose columns are indexed by $k$, where the pair $(i,j)$ corresponds to the $(i,j)$-th entry of the matrix representing a $(S_\alpha, S_\beta)$-double coset. When we consider $M^\alpha \otimes (M^\beta \otimes M^\gamma)$, $\theta_{ijk}$ is viewed as a matrix whose rows are indexed by $i$ and whose columns are indexed by pairs $(j,k)$ corresponding to a $(S_\beta, S_\gamma)$-double coset.
\end{remark}

\begin{proposition}
There is a basis $\xi_q$ of $S(n,d)$ indexed by double cosets $q \in S_{\alpha} \backslash S_d / S_{\beta}$, where $\alpha, \beta$ vary over all compositions of $d$ into $n$ parts (where parts may have size zero).
\end{proposition}
\begin{proof}
This is a special case of Theorem 3.4 of \cite{DJ}, where the calculation is done in the more general setting of Iwahori-Hecke algebras (and $q$-Schur algebras) rather than symmetric groups. Firstly we write the weight space decomposition (an isomorphism of $S_d$-modules):
\[
V^{\otimes d} = \bigoplus_{\alpha \models d} M^\alpha.
\]
Then, to calculate the $S_d$ commutant, we apply the Mackey formula for induced representations in Proposition \ref{schur_basis_theorem}.
\end{proof}


\noindent
The following formula can be found in Proposition 2.3 of \cite{SY}.
\begin{proposition}\label{composition_rule}
Let $r$ and $s$ be double-cosets for arbitrary Young subgroups of $S_d$. We have $\xi_{r} \xi_{s} = \sum_{q} C_{r, s}^{q} \xi_{q}$, where $C_{r, s}^{q}$ is defined as follows. Let $A = (a_{ijk})$ be a $n \times n \times n$ 3-tensor with entries in $\mathbb{Z}_{\geq 0}$. We require $A$ to satisfy
\begin{eqnarray*}
\sum_{k}a_{ijk} &=& r_{ij} \\
\sum_{j}a_{ijk} &=& q_{ik} \\
\sum_{i}a_{ijk} &=& s_{jk}
\end{eqnarray*}
and we let
\begin{equation} \label{composition_help}
V(A) = \prod_{i,k} \frac{(\sum_j a_{ijk})!}{\prod_j a_{ijk}!}
\end{equation}
Then, summing over $A$ as above,
\begin{equation} 
C_{r,s}^{q} = \sum_{A} V(A).
\end{equation}
\end{proposition}
\begin{proof}
Suppose that $r$ is a $(S_{\beta^\prime},S_\gamma)$-double coset representative, and $s$ is a $(S_\alpha,S_{\beta})$-double coset representative. Because $\beta_j^\prime = \sum_{i} r_{ij} = \sum_{i,k} a_{ijk} = \sum_{k} s_{jk} = \beta_j$, it immediately follows that if $\beta \neq \beta^\prime$, there can be no such $a_{ijk}$. In this case, the sum is empty and the formula correctly gives that $\xi_{r} \xi_{s}=0$.
\newline \newline \noindent
If $\beta = \beta^\prime$, we may use the formula from the proof of Proposition \ref{schur_basis_theorem} to identify the action of the product $\xi_{r} \xi_{s}$ on $M^\alpha$. To this end, we recall that  a $(S_\alpha, S_\beta)$-double coset was given by the data of how many elements in $\{1,2,\ldots,n\}$ in the subset permuted by the factor group $S_{\beta_j}$ are mapped to elements permuted by the factor group $S_{\alpha_k}$. Upon composing this with another map coming from a $(S_\beta, S_\gamma)$-double coset, an element permuted by the factor group $S_{\gamma_i}$ could be mapped to an element permuted by $S_{\alpha_k}$ via any intermediate factor group $S_{\beta_j}$. Suppose there are $a_{ijk}$ such elements, and let us restrict our attention to just the subset of $\{1,2,\ldots,n\}$ permuted by $S_{\gamma_i}$ that is mapped into the subset permuted by $S_{\alpha_k}$. Among the $(\sum_j a_{ijk})!$ possible maps arising from permutations, the latent $S_{\beta_j}$ actions render elements mapping through the elements permuted by $S_{\beta_j}$ indistinguishable, thus we must divide by $\prod_j a_{ijk}!$.
Taking the product over all possible $i$ and $k$, we obtain
\[
\prod_{i,k} \frac{(\sum_j a_{ijk})!}{\prod_j a_{ijk}!}.
\]
\end{proof}

\begin{proposition} \label{alt_pres}
An alternative description of $S(n,d)$ is as follows. Consider the set $A(n,d)$ of degree $d$ integer polynomials in the variables $x_{ij}$, where $1 \leq i, j \leq n$. This is a coalgebra with comultiplication defined by $\Delta(x_{ij}) = \sum_{k=1}^n x_{ik} \otimes x_{kj}$ and counit $\varepsilon(x_{ij}) = \delta_{ij}$ (extended multiplicatively to polynomials of degree $d$). The dual $\Hom_R(A(n,d), R)$ is isomorphic to the Schur algebra via the map sending the basis dual to $\prod_{i,j=1}^n x_{ij}^{q_{ij}}$ to $\xi_{q_{ij}}$.
\end{proposition}

\begin{proof}
This definition is used in Section 2.3 of \cite{Green}, where the structure constants are explicitly calculated, and agree with those given in Proposition \ref{composition_rule}.
\end{proof}
\noindent
In the special case where $R = \mathbb{Q}$, we will use the following theorem.


\begin{lemma} \label{eigenvalue_lemma}
Consider the one-dimensional space $\Hom_{\mathbb{Q}S_d}(S^\lambda, M^\lambda)$, which is a left module for $\End_{\mathbb{Q}S_d}(M^\lambda)$. Let $c_q$ be the scalar by which the element $\xi_{q}$ acts on this space. Then,
\[
\sum_{q} (\prod_{ij} x_{ij}^{q_{ij}}) c_{q} = \det(x_{11})^{\lambda_1-\lambda_2} \det
\left( {\begin{array}{cc}
   x_{11} & x_{12} \\
   x_{21} & x_{22}
  \end{array} } \right)^{\lambda_2 - \lambda_3} \cdots \det(X)^{\lambda_d}
\]
where $X$ is the matrix $(x_{ij})_{i,j=1}^d$. An identical formula holds for the right $\End_{\mathbb{Q}S_d}(M^\lambda)$-module $\Hom_{\mathbb{Q}S_d}(M^\lambda, S^\lambda)$.
\end{lemma}

\begin{proof}
At the start of Section 2.4 in \cite{Green}, it is explained that the group algebra $\mathbb{Z}GL(V)$ admits a surjective natural map to $S_{\mathbb{Z}}(n,d)$ as follows. We use Proposition \ref{alt_pres}, giving an element of $S_{\mathbb{Z}}(n,d)$ in the form of a linear functional on $A(n,d)$. Let $g = (y_{ij}) \in GL(V)$, then we take
\[
g \mapsto \left (\prod_{ij}x_{ij}^{q_{ij}} \mapsto \prod_{ij}y_{ij}^{q_{ij}} \right) = \sum_{q} \xi_q \prod_{ij} y_{ij}^{q_{ij}},
\]
and extend linearly to $\mathbb{Z}GL(V)$. The action of $\mathbb{Z}GL(V)$ on $V^{\otimes d}$ factors through this map.
\newline \newline \noindent
We may consider a version of the Bruhat decomposition of $GL(V)$ (where $V= \mathbb{Q}^{\oplus n}$):
\[
GL(V) = \coprod_{w \in S_n} B_-wB_+,
\]
where $B_-$ is the Borel subgroup of lower-triangular matrices, while $B_+$ is the Borel subgroup of upper-triangular matrices. On the largest Bruhat cell (corresponding to $w=\Id$), a matrix $X$ admits a factorisation $X=b_- b_+$ (where $b_- \in B_-$ and $b_+ \in B_+$), which we may refine to $X = LDU$, where $L$ is lower-unitriangular, $D$ is diagonal, and $U$ is upper-unitriangular. Moreover, by standard linear algebra, the diagonal entries $D_r$ of $D$ are given by the formula
\[
D_r = \frac{\det((x_{ij})_{i,j=1}^{r})}{\det((x_{ij})_{i,j=1}^{r-1})}.
\]
The action of $X = LDU$ on a highest-weight vector $v$ of weight $\lambda$ can be understood as follows. Firstly, $Uv = v$, because highest-weight vectors are $U$-invariant. The action of $D$ is
\[
Dv = \prod_r D_r^{\lambda_r}v,
\]
because $v$ has weight $\lambda$. Finally, the $L$ term only serves to add terms of lower weights. However, we were only interested in terms $\xi_{q_{ij}} \in \End_{S_d}(M^\lambda)$, so we may discard the vectors of lower weight. This proves the theorem on a Zariski-dense subset of $GL(V)$, hence we obtain the statement of the lemma. The case of right modules is identical.
\end{proof}

\subsection{The Deligne Category}
In this section we work over a commutative ground ring $R$ with a distinguished element $t \in R$. The two main cases of interest are $R = \mathbb{Q}$ with $t \in \mathbb{Z}_{\geq 0}$ and $R = \mathbb{Z}[t]$ (where the distinguished element is the variable $t$). We omit most of the proofs; they can be found in \cite{OC}.

\begin{definition}
A $(m,n)$-partition diagram is a set partition of the set $W_{m,n} = \{1,2,\ldots , m, 1^\prime, 2^\prime, \ldots, n^\prime\}$. We write ${\rm Par}_{m,n}$ for the set of these. We depict such a partition diagram by means of a graph whose vertices are labelled by the set $W_{m,n}$ and whose connected components are the parts of the set partition (although in principle there are many choices of graphs with the same connected components, they are all equivalent for our purposes). For convenience we arrange the vertices into two rows, the first consisting of the unprimed vertices and the second consisting of the primed vertices.
\end{definition}

\begin{example} \label{partition_diagram_example}
The diagram below represents the set partition $\{\{1,2,1^\prime\}, \{3, 2^\prime\}, \{4\}\} \in Par_{4,2}$.
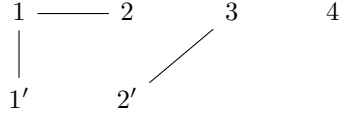
\[
\begin{tikzcd}
1\arrow[r, dash]\arrow[d, dash]& 2 & 3\arrow[ld, dash] & 4\\
1^\prime & 2^\prime
\end{tikzcd}
\]
\end{example}

\begin{definition}
The pre-Deligne category $\underline{{\rm Rep}}_0(S_t)$ is defined as follows. The objects $[n]$ are indexed by $n \in \mathbb{Z}_{\geq 0}$. The morphisms $\Hom([m], [n])$ are $R$-linear combinations of $(m,n)$-partition diagrams. Composition of morphisms is $R$-linear, and the composition of a $(m,n)$-partition diagram $D_1$ with a $(k,m)$-partition diagram $D_2$ is obtained by the following process. Consider the labelled graphs associated to the two partition diagrams; identify the vertices $1,2, \ldots, m$ in the first partition diagram with the elements $1^\prime, 2^\prime, \ldots, m^\prime$ in the second partition diagram, retaining all edges. This yields the graph of a $(k,n)$-partition diagram $D_3$ (containing some of the intermediate vertices which were identified) together with $r$ connected components consisting only of vertices which were identified. We define $D_1 \circ D_2 = t^{r} D_3$.
\end{definition}

\begin{example} \label{partition_composition_example}
We demonstrate composition of morphisms in the pre-Deligne category by composing the partition $\{\{1,2,1^\prime\}, \{3, 2^\prime\}, \{4\}\} \in {\rm Par}_{4,2}$ from Example \ref{partition_diagram_example} with the partition $\{\{1, 1^\prime\}, \{2, 3\}, \{2^\prime, 3^\prime\}, \{4^\prime\}\} \in {\rm Par}_{3,4}$, which corresponds to the following diagram:
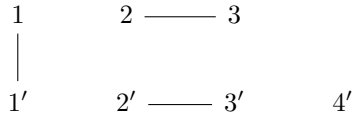
\[
\begin{tikzcd}
1\arrow[d, dash]& 2\arrow[r, dash] & 3\\
1^\prime & 2^\prime\arrow[r, dash] & 3^\prime & 4^\prime
\end{tikzcd}.
\]
We now show the two partitions so that the vertices to be merged are vertically adjacent to each other.
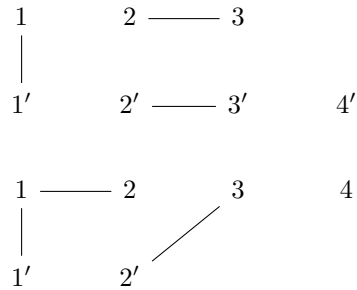
\[
\begin{tikzcd}
1\arrow[d, dash]& 2\arrow[r, dash] & 3\\
1^\prime & 2^\prime\arrow[r, dash] & 3^\prime & 4^\prime \\
1\arrow[r, dash]\arrow[d, dash]& 2 & 3\arrow[ld, dash] & 4\\
1^\prime & 2^\prime
\end{tikzcd}
\]
Merging the appropriate vertices gives the following intermediate result, where we label the merged vertices with double primes.
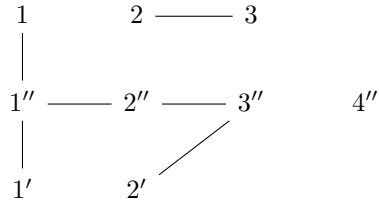
\[
\begin{tikzcd}
1\arrow[d, dash]& 2\arrow[r, dash] & 3\\
1^{\prime \prime}\arrow[r, dash]\arrow[d, dash]& 2^{\prime \prime}\arrow[r, dash] & 3^{\prime \prime}\arrow[ld, dash] & 4^{\prime \prime}\\
1^\prime & 2^\prime
\end{tikzcd}
\]
Noting that we have one connected component not connected to either $1,2,3$ or $1^\prime,2^\prime$ (namely, $\{4^{\prime \prime}\}$), we obtain a scalar factor of $t^1$ when we pass to the induced diagram on $\{1,2,3,1^\prime,2^\prime\}$. Our final result is
\[
t\hspace{2mm} \cdot \hspace{2mm}
\begin{tikzcd}
1\arrow[d, dash]& 2\arrow[r, dash] & 3\\
1^\prime \arrow[r, dash] & 2^\prime
\end{tikzcd}.
\]
\end{example}

\begin{proposition}
There is a symmetric monoidal structure on $\underline{{\rm Rep}}_0(S_t)$ given by $[n]\otimes [m]= [n+m]$ on objects, with the action on morphisms given by concatenation (and relabelling) as follows. If $D_1 \in {\rm Par}_{m_1,n_1}$ and $D_2 \in {\rm Par}_{m_2,n_2}$, then $D_1 \otimes D_2 \in {\rm Par}_{m_1+m_2, n_1+n_2}$ is the diagram where vertices with labels in $\{1,2,\ldots,m_1, 1^\prime, 2^\prime, \ldots, n_1^\prime\}$ have the same connections as $D_1$, while vertices with labels in 
\[
\{m_1+1,m_1+2,\ldots,m_1+m_2, (n_1+1)^\prime, (n_1+2)^\prime, \ldots, (n_1+n_2)^\prime\}
\]
have the same connections as $D_2$ (in the obvious way; biject the vertices to those of $D_2$ by subtracting $m_1$ from vertex labels without primes, and $n_1$ from vertex labels with primes). No edges are included between the vertices corresponding to $D_1$ and the vertices corresponding to $D_2$.
\end{proposition}

\begin{remark}
In fact, $\underline{{\rm Rep}}_0(S_t)$ is a rigid symmetric monoidal category, but we will not use this fact. Again, we direct the interested reader to Section 2 of \cite{OC}.
\end{remark}

\begin{definition}
The Deligne category $\underline{{\rm Rep}}(S_t)$ is the Karoubian envelope of the pre-Deligne category $\underline{{\rm Rep}}_0(S_t)$. It inherits the structure of a (rigid) symmetric monoidal category.
\end{definition}

\begin{definition}
Suppose that the distinguished element $t \in R$ is set to be $d \in \mathbb{Z}_{\geq 0}$. There is a symmetric monoidal functor $F_d$ from $\underline{{\rm Rep}}(S_d)$ (the Deligne category with $t=d$) to $R S_d-{\rm mod}$, the category of finitely generated representations of the symmetric group $S_d$ over $R$. This functor is defined first for $\underline{{\rm Rep}}_0(S_t)$, and then we pass to the Karoubian envelope. Let $V$ be the permutation representation of $S_d$ on $R^{\oplus d}$. Then, $F_d([m]) = V^{\otimes m}$. To define the action of $F_n$ on morphisms, let $v_1, v_2, \ldots, v_n$ be the standard basis of $V$. Then a diagram $D \in {\rm Par}_{p,q}$ acts on the pure tensor $v_{i_1} \otimes v_{i_2} \otimes \cdots \otimes v_{i_p}$ to give the sum of all $v_{i_{1^\prime}} \otimes v_{i_{2^\prime}} \otimes \cdots \otimes v_{i_{q^\prime}}$ such that if two vertices $x, y$ (possibly with or without primes) are in the same component of $D$, then $i_x = i_y$. 
\end{definition}

\begin{remark}
Taking the $R$-span of ${\rm Par}_{n,n}$ for a fixed $n \in \mathbb{Z}_{\geq 0}$, we obtain an associative algebra (depending on the parameter $t$) using this multiplication. These algebras are called the partition algebras, and are often denoted ${\rm Par}_n(t)$. By the above,  we have a map ${\rm Par}_n(d) \to \End_{S_d}(V^{\otimes n})$. This has been used to study an analogue of Schur-Weyl duality between symmetric group algebras and partition algebras. It allows for the study of the representations of one algebra in terms of the other. This has been the subject of much work in recent years; for example, \cite{BHH} and \cite{BDO}.
\end{remark}

\begin{example}
We consider the partition $\{\{1, 1^\prime\}, \{2, 3\}, \{2^\prime, 3^\prime\}, \{4^\prime\}\} \in {\rm Par}_{3,4}$ (from Example \ref{partition_composition_example}), whose diagram is
\[
\begin{tikzcd}
1\arrow[d, dash]& 2\arrow[r, dash] & 3\\
1^\prime & 2^\prime\arrow[r, dash] & 3^\prime & 4^\prime
\end{tikzcd}.
\]
This acts on $v_{i_1} \otimes v_{i_2} \otimes v_{i_3}$ to give
\[
\delta_{i_2, i_3} \sum_{i_{2^\prime}=1}^n \sum_{i_{4^\prime}=1}^n v_{i_1} \otimes v_{i_{2^\prime}} \otimes v_{i_{2^\prime}} \otimes v_{i_{4^\prime}}
\]
where we point out that unless $i_2=i_3$, the required condition on the component $\{2,3\}$ is not satisfied. Similarly, only terms where the tensor factors corresponding to $2^\prime$ and $3^\prime$ have the same label arise.
\end{example}

\begin{proposition}
Suppose that we take $R$ to be a field of characteristic zero. Then the indecomposable objects $X_\mu$ are indexed by the set of all partitions $\mu$. Further, $F_d(X_\mu)$ is equal to $S^{(d - |\mu|, \mu)}$ if $d - |\mu| \geq \mu_1$ (so that this defines a valid partition), and $F_d(X_\mu) = 0$ otherwise.
\end{proposition}

\noindent
The key fact relating Deligne categories to representations of symmetric groups is the following proposition.

\begin{proposition}
When $R = \mathbb{C}$ (or any field of characteristic zero), $F_d$ is essentially surjective and full.
\end{proposition}

\begin{definition}
Suppose that $D$ and $E$ are set partitions of a set $X$.
For $D \in {\rm Par}_{p,q}$, define $x_D \in R \hspace{1mm} {\rm Par}_{p,q}$ via the recurrence 
\[
D = \sum_{\mbox{$E$ a coarsening of $D$}} x_E.
\]
\end{definition}
\noindent
Thus, the $x_D$ can be calculated by a recursion over the poset of partition diagrams, ordered by coarsening. In particular, they also form a basis of morphism spaces (the transition matrix to this basis is upper unitriangular). The following proposition about how the $x_D$ act on $V^{\otimes p}$ follows from an inclusion-exclusion argument.
\begin{proposition} \label{x_D_action}
The element $F_d(x_D)$ acts on the pure tensor $v_{i_1} \otimes v_{i_2} \otimes \cdots \otimes v_{i_p}$ to give the sum of all $v_{i_{1^\prime}} \otimes v_{i_{2^\prime}} \otimes \cdots \otimes v_{i_{q^\prime}}$ such that two vertices $x, y$ (possibly with or without primes) are in the same component of $D$ if and only if $i_x = i_y$. 
\end{proposition}
\noindent
The key detail of this proposition is that distinct components in the partition diagram of $D$ are required to have distinct indices labelling them. Note that for $S_d$, there can be at most $d$ distinct labels. This explains the following fact.
\begin{proposition}
The kernel of $F_d$ on $\Hom([m],[n])$ is precisely the span of $x_D$ corresponding to $D$ with more than $d$ connected components. 
\end{proposition}
\noindent
We now work towards relating the Deligne category, as defined above, to the Schur algebra setting which will be important for us.

\begin{proposition} \label{tensor_pow_dec}
Let us write 
\[
V= \Ind_{S_{d-1} \times S_1}^{S_d} (R) = M^{(d-1,1)},
\]
where $R$ is the trivial representation of $S_{d-1}\times S_1$; $V$ is the usual permutation representation of $S_d$ on $R^{\oplus d}$. Then we may decompose $V^{\otimes r}$ as a sum of permutation modules in the following way:
\[
V^{\otimes r} = 
\bigoplus_{
\substack{
\mbox{{\rm set partitions} $\alpha$ {\rm of} $\{1,2,\ldots,r\}$} \\
\mbox{{\rm into at most $d$ parts}}
}} N^\alpha,
\]
where $N^\alpha$ is isomorphic to $M^{(d-l(\alpha), 1^{l(\alpha)})}$.
\end{proposition}

\begin{proof}
For any set partition $\alpha$ as above, consider the span of pure tensors $v_{i_1} \otimes v_{i_2} \otimes \cdots \otimes v_{i_r}$, such that $i_p=i_q$ if and only if $p$ and $q$ are in the same part of $\alpha$. Since these vectors are permuted transitively by the implied $S_n$ action (which changes labels of tensor factors, but not positions), we obtain a permutation module. Each of the $l(\alpha)$ parts of $\alpha$ must correspond to a distinct basis vector, of which there are $d$ to choose from. Hence, the module is induced from the trivial representation of $S_{(d-l(\alpha), 1^{l(\alpha)})}$, and is therefore isomorphic to $M^{(d-l(\alpha), 1^{l(\alpha)})}$.
\end{proof}
\noindent
We now seek to understand maps between the summands in this direct sum decomposition. One approach is via Proposition \ref{schur_basis_theorem}. Since $N^\alpha$ and $N^\beta$ are permutation modules induced from Young subgroups, a basis of homomorphisms between them is given by elements $\xi_q$ indexed by $(S_{(d-l(\alpha),1^{l(\alpha))}},S_{(d-l(\beta),1^{l(\beta)})})$-double cosets. Another approach is given by partition diagrams. We now relate partition diagrams to the $\xi_q$.

\begin{definition} \label{D_definition}
Let $\alpha$ and $\beta$ be set partitions of $\{1,2,\ldots,r\}$ and $\{1,2,\ldots,s\}$, respectively, each into at most $n$ parts. Let $q$ be a $(S_{(n-l(\alpha), 1^{l(\alpha)})}, S_{(n-l(\beta),1^{l(\beta)})})$-double coset in $S_d$. We define an $(r,s)$-partition diagram $D(q)$ in the following way. Note that $q$ is uniquely specified by a matching (injective partial function) from the parts of $\alpha$ to the parts of $\beta$; viewing $q = (q_{ij})$ as a matrix, each row except the first and each column except the first must sum to one. So, all their entries except for exactly one are zero. The coordinates of the nonzero entries that are not in the first row or first column define the matching. For $1 \leq i, j \leq r$, let $i$ and $j$ be in the same connected component of $D(q)$ if and only if $i$ and $j$ are in the same part of $\alpha$. Similarly for $1 \leq i,j \leq s$, let $i^\prime$ and $j^\prime$ be in the same connected component of $D(q)$ if and only if $i$ and $j$ are in the same part of $\beta$. Finally, we connect the connected components corresponding to the part $\alpha_i$ and $\beta_j$ if and only if the corresponding entry of $q$ is equal to 1. (That is, $q_{i+1,j+1}=1$, taking into account that the first row and first column of $q$ do not correspond to any part of $\alpha$ or $\beta$).
\end{definition}

\begin{example}
Suppose that $q$ is the $(S_{8} \times S_{1}^2, S_{7} \times S_{1}^3)$-double coset indexed by the following matrix.
\[
\left( \begin{array}{ccc}
6 & 0 & 1 \\
1 & 0 & 0\\
0 & 1 & 0\\
1 & 0 & 0
\end{array} \right).
\]
Then, $D(q)$ is the following partition diagram.
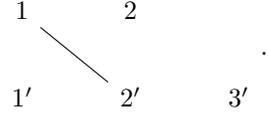
\[
\begin{tikzcd}
1\arrow[dr, dash]& 2 \\
1^\prime & 2^\prime & 3^\prime
\end{tikzcd}.
\]
\end{example}

\begin{proposition}
Given a $(S_{(d-l(\alpha), 1^{l(\alpha)})}, S_{(d-l(\beta),1^{l(\beta)})})$-double coset $q$ in $S_d$, the two maps $N^\alpha \to N^\beta$ defined by $\xi_q$ and $F_n(x_{D(q)})$ are equal.
\end{proposition}

\begin{proof}
To understand how $\xi_q$ acts on $N^\alpha \subseteq V^{\otimes r}$, we recall that $N^\alpha$ has a basis consisting of pure tensors where the $i$-th and $j$-th tensor factors are the same basis vector of $V$ if and only if $i$ and $j$ are in the same part of $\alpha$. We may therefore associate a basis vector $v_{I(\alpha_i)}$ to each part $\alpha_i$ of the set partition $\alpha$. We write the pure tensor as
\[
\bigotimes_{i=1}^{l(\alpha)} v_{I(\alpha_i)}^{\otimes \alpha_i}
\]
where a vector $v$ raised to a tensor power that is a part of a set partition indicates that the tensor factors corresponding to the elements in that part equal to the $v$. (For example $v^{\otimes \{1,3\}} \otimes w^{\otimes \{2\}} = v \otimes w \otimes v$.) A similar description holds for $N^\beta$:
\[
\bigotimes_{j=1}^{l(\beta)} v_{J(\beta_j)}^{\otimes \beta_j}.
\]
Here the symmetric group action is by permuting the indices of the tensor factors (rather than permuting the tensor factors themselves as in the setting of Schur-Weyl duality).
\newline \newline \noindent
Recall that the action of $\xi_q$ can be understood by taking the sum of the actions of all $S_{(d-l(\beta),1^{l(\beta)})}$ coset representatives in the double coset $S_{(d-l(\alpha), 1^{l(\alpha)})}q S_{(d-l(\beta),1^{l(\beta)})}$. An element $\sigma$ of the double coset $S_{(d-l(\alpha), 1^{l(\alpha)})}q S_{(d-l(\beta),1^{l(\beta)})}$ acts on our pure tensor to produce a pure tensor as follows:
\[
\sigma \left( \bigotimes_{i=1}^{l(\alpha)} v_{I(\alpha_i)}^{\otimes \alpha_i} \right)
=
\bigotimes_{j=1}^{l(\beta)} v_{\sigma(J(\beta_j))}^{\otimes \beta_j}.
\]
Here $\sigma(J(\beta_j)) = I(\alpha_i)$ if and only if $q_{i+1,j+1}=1$ (because this means the double coset matches $\alpha_i$ with $\beta_j$), but otherwise $\sigma$ can be any permutation. This means that for indices $j$ such that $q_{ij}=0$ for all $i \neq 1$ (meaning that the $1$ in column $j$ appears in the first row), $\sigma(J(\beta_j))$ can take any value different from the $I(\alpha_i)$. However, as long as these values are distinct for different $j$ (as $\sigma$ is a permutation), there are no restrictions. The action of $\xi_q$ is given by the sum of all such choices. However this is, by construction, precisely the action of $F_n(x_{D(q)})$ according to Proposition \ref{x_D_action}.
\end{proof}
\noindent
Suppose that $\alpha, \beta, \alpha^\prime, \beta^\prime$ are compositions of $d$. We may write
\[
M^\alpha \otimes M^\beta = \bigoplus_{\gamma \in S_\alpha \backslash S_d / S_\beta} M^{\gamma}
\]
and
\[
M^{\alpha^\prime} \otimes M^{\beta^\prime} = \bigoplus_{\gamma^\prime \in S_{\alpha^\prime} \backslash S_d / S_{\beta^\prime}} M^{\gamma^\prime}.
\]
Here we have used the convention that a double coset $\gamma$ (or $\gamma^\prime$) may be considered as a composition whose parts are the entries of the associated matrix. We may represent $\gamma$ by a $\mathbb{Z}_{\geq 0}$-valued matrix $\gamma_{ij}$ such that the $i$-th row sums to $\beta_i$ and the $j$-th column sums to $\alpha_j$. Similarly we may represent $\gamma^\prime$ by a matrix $\gamma_{kl}^\prime$ whose row sums are $\beta_k^\prime$ and whose column sums are the parts of $\alpha_l^\prime$. Suppose that $q^{(1)}$ is a $(S_\alpha, S_{\alpha^\prime})$-double coset, and $q^{(2)}$ is a $(S_\beta, S_{\beta^\prime})$-double coset. We indicate a $(S_\gamma,S_{\gamma^\prime})$-double coset by a $4$-tensor $T_{ijkl}$ where $\sum_{kl} T_{ijkl} = \gamma_{ij}^\prime$ and $\sum_{ij} T_{ijkl} = \gamma_{kl}$.
\begin{proposition} \label{tensor_rule}
The morphism 
\[
\xi_{q^{(1)}} \otimes \xi_{q^{(2)}} : M^\alpha \otimes M^\beta \to M^{\alpha^\prime} \otimes M^{\beta^\prime}
\]
restricts as a map $f: M^{(\gamma)} \to M^{(\gamma^\prime)}$ in the following way.
\[
\xi_{q^{(1)}} \otimes \xi_{q^{(2)}} \restriction_{M^\gamma \to M^{\gamma^\prime}}
= \sum_{T_{ijkl}}  \xi_{T},
\]
where the sum ranges over all $4$-tensors $T_{ijkl}$ subject to
\begin{eqnarray*}
\sum_{kl} T_{ijkl} &=& \gamma_{ij}^\prime\\
\sum_{ij} T_{ijkl} &=& \gamma_{kl}\\
\sum_{ik} T_{ijkl} &=& q_{jl}^{(2)}\\
\sum_{jl} T_{ijkl} &=& q_{ik}^{(1)}.
\end{eqnarray*}
\end{proposition}

\begin{proof}
We describe $M^\gamma \subseteq M^\alpha \otimes M^\beta$ as the span of pure tensors of the form
\[
(v_{r_1} \otimes v_{r_2} \otimes \cdots \otimes v_{r_n}) \otimes (v_{s_1} \otimes v_{s_2} \otimes \cdots \otimes v_{s_n})
\]
where the number of pairs $(r_m,s_m)$ equal to $(i,j)$ is precisely $\gamma_{ij}$. Then, $\xi_{q^{(1)}}$ acts on the first $n$ tensor factors, while $\xi_{q^{(2)}}$ acts on the second $n$ tensor factors. If we have a pair of indices $(r_m,s_m)$ equal to $(i,j)$, they may be mapped to $(k,l)$ by $\xi_{q^{(1)}}$ mapping the relevant $v_i$ to $v_k$ and by $\xi_{q^{(2)}}$ mapping the relevant $v_j$ to $v_l$; if there are $T_{ijkl}$ values of $m$ for which this happens, we obtain $\xi_T$.

\end{proof}

\section{Categorical Interpolation and Specialisation}
\noindent
In this section we construct categories that interpolate the composition rule in Proposition \ref{composition_rule} and tensor product rule in Proposition \ref{tensor_rule}. In our setting, the parts of a partition $\lambda$ will be considered as variables. Hence, we work over the ring of integer-valued polynomials in the parts of $\lambda$, as per the following definition. As before, $\lambda_r$ is taken to be zero when $r > l(\lambda)$.

\begin{definition}
Let $\mathcal{R}$ be the ring of integer-valued polynomials in one variable (free as a $\mathbb{Z}$-module with binomial coefficients ${x \choose n}$ as a basis). Write $\mathcal{R}_l$ for the ring of integer-valued polynomials in the variables $\lambda_1, \lambda_2, \ldots, \lambda_l$. It is canonically isomorphic to $\mathcal{R}^{\otimes l}$. Given a composition of length $l$, we define $\phi_\mu : \mathcal{R}_l \to \mathbb{Z}$ to be the homomorphism which evaluates $\lambda$ at $\mu$.
\end{definition}

\noindent
To interpolate morphisms $\xi_q$ between $M^\alpha$ and $M^\beta$ (as modules for $S_d$), we will require that $\alpha$ and $\beta$ only deviate ``finitely'' from our parameter partition $\lambda$; to be precise, we assume $\alpha_i, \beta_i \in \lambda_i + \mathbb{Z}$ for all $i \leq l(\lambda)$, and $\alpha_i, \beta_i \in \mathbb{Z}_{\geq 0}$ otherwise. Note that this means that the parts of $\alpha$ and $\beta$ depend on the variables $\lambda_i$, hence such $\alpha$ and $\beta$ are not integer partitions.

\begin{definition}
Let $T_\lambda$ be the set of sequences of the form 
\[
(\lambda_1 + \sigma_1, \lambda_2 + \sigma_2, \ldots, \lambda_l + \sigma_l, \tau_1, \tau_2, \ldots) =
(\lambda + \sigma, \tau),
\]
where $\sigma \in \mathbb{Z}^l$ and $\tau$ is a composition such that $\sum_{i=1}^l \sigma_i + |\tau| = 0$. We treat elements of $T_\lambda$ as vectors so that it makes sense to add componentwise. Additionally, given a composition $\mu$ of length $l$, extend $\phi_\mu$ to elements of $T_\lambda$ componentwise (the result will be an integer vector with sum of coordinates $|\mu|$).
\end{definition}
\noindent
Although elements of $T_\lambda$ are not partitions, we can still make sense of $(S_\alpha, S_\beta)$-double cosets as in Proposition \ref{young_double_cosets}.
\begin{definition}
Suppose that ${\alpha}, {\beta} \in T_\lambda$. Let $T_\lambda({\alpha}, {\beta})$ be the set of matrices $A$ satisfying the following properties. The size of $A$ is $l({\beta}) \times l({\alpha})$. For $1 \leq i \leq l(\lambda)$, the $i$-th diagonal entry satisfies $A_{ii} \in \lambda_i + \mathbb{Z}$, while all other entries must be non-negative integers. We require the row sums to be ${\beta}$ and the column sums to be ${\alpha}$.
\end{definition}
\noindent
These elements index bases of morphism spaces in the category that we construct. We modify the composition rule in Proposition \ref{composition_rule} to our setting.

\begin{definition} \label{enhanced_composition_rule}
Let ${r} \in T({\beta}, {\gamma})$, ${s} \in T_\lambda({\alpha}, {\beta})$ and let ${q} \in T({\alpha}, {\gamma})$. We define the product of the symbols $\xi_{{r}}$ and $\xi_{{s}}$ to be the formal linear combination of symbols $\xi_{{q}}$
\[
\xi_{{r}}\xi_{{s}}
=
\sum_{{q} \in T_\lambda({\alpha}, {\gamma})}
C_{{r},{s}}^{{q}} \xi_{{q}}
\]
given as follows. Let $A = (a_{ijk})$ be a 3-tensor (of size $l({\gamma})\times l({\beta})\times l({\alpha})$) such that for $1 \leq i \leq l(\lambda)$, $a_{iii} \in \lambda_i + \mathbb{Z}$ and all other entries are in $\mathbb{Z}_{\geq 0}$. We require $A$ to satisfy
\begin{eqnarray*}
\sum_{k}a_{ijk} &=& {r}_{ij} \\
\sum_{j}a_{ijk} &=& {q}_{ik} \\
\sum_{i}a_{ijk} &=& {s}_{jk}
\end{eqnarray*}
and we let
\begin{equation*}
V(A) = \prod_{i,k} \frac{(\sum_j a_{ijk})!}{\prod_j a_{ijk}!}
\end{equation*}
Then, summing over such $A$,
\begin{equation*} 
C_{{r},{s}}^{{q}} = \sum_{A} V(A).
\end{equation*}
\end{definition}

\begin{proposition} \label{enhanced_properties}
We have the following properties of Definition \ref{enhanced_composition_rule}.
\begin{enumerate}
\item The sum $\sum_{{q}} C_{{r},{s}}^{{q}} \xi_{{q}}$ has finitely many terms.
\item The coefficients $C_{{r},{s}}^{{q}}$ are elements of $\mathcal{R}_l$.
\item Let $\phi_\mu: \mathcal{R}_l \to \mathbb{Z}$ be defined by evaluating $\lambda$ at some integer partition. Also let $\phi_\mu({r})$ be the matrix obtained from $({r}_{ij})$ by applying $\phi_\mu$ to its entries. Further let $\phi_\mu(\xi_{{r}})$ be $\xi_{\phi_\mu(r)}$ (an element of the Schur algebra) if $\phi_\mu(r)$ has non-negative entries, and let it be zero otherwise. Then $\phi_\mu$ intertwines the composition in Definition \ref{enhanced_composition_rule} and the composition in Proposition \ref{composition_rule}.
\item The composition rule in Definition \ref{enhanced_composition_rule} is associative.
\end{enumerate}
\end{proposition}

\begin{proof}
\noindent
Consider the possible matrices $A = (a_{ijk})$ that arise. Because there are finitely many off-diagonal entries ($i,j,k$ not all equal), each of which is bounded either by ${r}_{ij}$ or ${s}_{jk}$ (at least one of which is a non-negative integer), there are finitely many possibilities for the off-diagonal entries. The diagonal entries are determined by the off-diagonal entries using the row-sum conditions. This means there are only finitely many such $A$ possible. This proves the sum is finite.
\newline \newline \noindent
Each $V(A)$ is a product of multinomial coefficients, where in each multinomial coefficient, at most one parameter is not a non-negative integer (i.e. lies in $\lambda_i + \mathbb{Z}$ for some $i$, which can only happen when the indices of $a_{ijk}$ are all equal). Such multinomial coefficients lie in $\mathcal{R}_l$. This guarantees that the $C_{{r},{s}}^{q}$ lie in $\mathcal{R}_l$.
\newline \newline \noindent
Because the two composition rules are almost identical, what must be shown is precisely that if one of $\phi_\mu({r})$ or $\phi_\mu({s})$ has a negative entry, then $\phi_\mu(C_{{r},{s}}^{{q}}) = 0$ for $q$ such that $\phi_\mu({q})$ has no negative entries. Suppose that $\varphi_\mu(r)$ has a negative entry (necessarily on the diagonal), $\phi_\mu({r}_{dd}) < 0$, and consider any matrix $A = (a_{ijk})$ arising in the calculation. Only diagonal entries of $\phi_\mu(A_{ijk})$ can be negative, because off-diagonal entries are non-negative integers, so the row-sum condition implies $\phi_\mu(a_{ddd}) < 0$. The numerator of $V(A)$ is the product of the factorials of the entries of ${q}$ (yielding some positive integer), while the denominator contains $a_{ddd}!$ which is (formally) infinite, so $V(A)=0$. The case where $s$ has a negative entry is identical.
\newline \newline \noindent
Fix ${r},{s},{q}$ and consider the quantity $(\xi_{{r}} \xi_{{s}}) \xi_{{q}} - \xi_{{r}} (\xi_{{s}} \xi_{{q}})$ as computed by Definition \ref{enhanced_composition_rule}. If this was not zero, it would contain a term $C\xi_{{p}}$ (where $C \in \mathcal{R}_l$). Consider $\phi_\mu$ specialising $\lambda$ to a partition $\mu$ with $\mu_1 \geq \mu_2 \geq \cdots \geq \mu_l \gg 0$ so that $\phi_\mu({p})$ has nonnegative entries, and hence $\xi_{\phi_\mu({p})}$ is not zero. But $(\xi_{\phi_\mu({r})} \xi_{\phi_\mu({s})}) \xi_{\phi_\mu({q})} - \xi_{\phi_\mu({r})} (\xi_{\phi_\mu({s})} \xi_{\phi_\mu({q})}) = 0$ because composition of morphisms of representations of $S_{|\mu|}$ is associative. This leads us to conclude that $\phi_\mu(C) = 0$. But the set of $\mu$ that we could choose is Zariski-dense, and we conclude that $C=0$.
\end{proof}

\noindent
We are now able to define the categories we are interested in.
\begin{definition}
Let $\mathcal{C}_\lambda^{(0)}$ be the $\mathcal{R}_l$-linear category whose objects $M^{{\alpha}}$ are indexed by ${\alpha} \in T_\lambda$. The morphisms between two objects are given by:
\[
\Hom_{\mathcal{C}_\lambda^{(0)}}(M^{{\alpha}}, M^{{\beta}}) = \mathcal{R}_l  T({\alpha}, {\beta})
\]
The composition of morphisms is $\mathcal{R}_l$-linear, given by Definition \ref{enhanced_composition_rule}. Now define $\mathcal{C}_\lambda$ to be the Karoubian envelope of $\mathcal{C}_\lambda^{(0)}$.
\end{definition}
\noindent
Due to Proposition \ref{enhanced_properties}, we know that composition is associative. However, we also obtain a specialisation functor to representations of symmetric groups.
\begin{definition}
Let $\mu \models n$ be a composition of length $l$ and $F_\mu^{(0)}$ be the functor from $\mathcal{C}^{(0)}_\lambda$ to the category of representations of $S_n$ defined by sending $M^{{\alpha}}$ to the permutation module $M^{\phi_\mu({\alpha})}$ if $\phi_\mu({\alpha})$ is a composition (i.e. has no negative entries), and zero otherwise. On morphisms, $F_\mu^{(0)}(C \xi_{{q}} ) = \phi_\mu(C) \xi_{\phi_\mu({q})}$. Let $F_\mu$ be the induced functor on Karoubian envelopes.
\end{definition}
\noindent
Part $(3)$ of Proposition \ref{enhanced_properties} shows that $F_\mu^{(0)}$ respects composition of morphisms, and hence defines a functor.



\begin{example}\label{chevalley_reln}
Take $l=2$, so that $\lambda = (\lambda_1, \lambda_2)$. Let ${e}$ be defined by the matrix
\[
\left( \begin{array}{cc}
\lambda_1 & 1 \\
0 & \lambda_2-1
\end{array} \right),
\]
and let ${f}$ be defined by
\[
\left( \begin{array}{cc}
\lambda_1 & 0 \\
1 & \lambda_2 - 1
\end{array} \right).
\]
Then to calculate $\xi_{{f}} \xi_{{e}}$ according to Definition \ref{enhanced_composition_rule}, we consider $2\times 2 \times 2$ tensors $A = (a_{ijk})$ satisfying certain sum conditions. We express these in tables, labelling each row or column by the relevant row-sum or column-sum constraint on $a_{i1k}$ and $a_{i2k}$.
\begin{center}
\begin{tabular}{ cc|c| } 
 \multicolumn{1}{l|}{$a_{i1k}$} & $\lambda_1$ & $1$ \\  \hline
 \multicolumn{1}{l|}{$\lambda_1$} & $a_{111}$ & $a_{112}$ \\  \hline
 \multicolumn{1}{l|}{$1$} & $a_{211}$ & $a_{212}$ \\ 
 \hline
\end{tabular} \hspace{10mm}
\begin{tabular}{ cc|c| } 
\multicolumn{1}{l|}{$a_{i2k}$} & $0$ & $\lambda_2-1$ \\  \hline
 \multicolumn{1}{l|}{$0$} & $a_{121}$ & $a_{122}$ \\  \hline
 \multicolumn{1}{l|}{$\lambda_2-1$} & $a_{221}$ & $a_{222}$ \\ 
 \hline
\end{tabular}
\end{center}
One easily sees that the only nonzero entry of $a_{i2k}$ is $a_{222}=\lambda_2 - 1$, but there are two possibilities for $a_{i1k}$. They are:
\[
\left( \begin{array}{cc}
\lambda_1 & 0 \\
0 &  1
\end{array} \right),\hspace{10mm}
\left( \begin{array}{cc}
\lambda_1-1 & 1 \\
1 & 0
\end{array} \right).
\]
Combining the two layers (according to the rule), we obtain
\[
{r} = \left( \begin{array}{cc}
\lambda_1 & 0 \\
0 &  \lambda_2
\end{array} \right), \hspace{10mm} {s} =
\left( \begin{array}{cc}
\lambda_1-1 & 1 \\
1 & \lambda_2-1
\end{array} \right).
\]
where the former term has a scalar multiple of $\frac{\lambda_2 !}{(\lambda_2-1)! 1!} = \lambda_2$ associated to it. Thus,
\[
\xi_{{f}} \xi_{{e}} = \lambda_2 \xi_{{r}} + \xi_{{s}}.
\]
Now instead consider ${e}^\prime$ be defined by the matrix
\[
\left( \begin{array}{cc}
\lambda_1-1 & 1 \\
0 & \lambda_2
\end{array} \right),
\]
and let ${f}^\prime$ be defined by
\[
\left( \begin{array}{cc}
\lambda_1-1 & 0 \\
1 & \lambda_2
\end{array} \right).
\]
Then to calculate $\xi_{{e}^\prime} \xi_{{f}^\prime}$, we perform the same steps. Our $a_{ijk}$-tables are
\begin{center}
\begin{tabular}{ cc|c| } 
 \multicolumn{1}{l|}{$a_{i1k}$} & $\lambda_1-1$ & $0$ \\  \hline
 \multicolumn{1}{l|}{$\lambda_1-1$} & $a_{111}$ & $a_{112}$ \\  \hline
 \multicolumn{1}{l|}{$0$} & $a_{211}$ & $a_{212}$ \\ 
 \hline
\end{tabular} \hspace{10mm}
\begin{tabular}{ cc|c| } 
 \multicolumn{1}{l|}{$a_{i2k}$} & $1$ & $\lambda_2$ \\  \hline
 \multicolumn{1}{l|}{$1$} & $a_{121}$ & $a_{122}$ \\  \hline
 \multicolumn{1}{l|}{$\lambda_2$} & $a_{221}$ & $a_{222}$ \\ 
 \hline
\end{tabular}
\end{center}
Similarly to before, $a_{i1k}$ must be zero except for $a_{111} = \lambda_1-1$, while $a_{i2k}$ may be either of:
\[
\left( \begin{array}{cc}
1 & 0 \\
0 & \lambda_2
\end{array} \right), \hspace{10mm}
\left( \begin{array}{cc}
0 & 1 \\
1 & \lambda_2-1
\end{array} \right).
\]
This leads to
\[
\xi_{{e}^\prime} \xi_{{f}^\prime} = \lambda_1 \xi_{{r}} + \xi_{{s}}.
\]
Finally, we note that
\[
\xi_{{e}^\prime} \xi_{{f}^\prime} - \xi_{{f}} \xi_{{e}} = (\lambda_1 - \lambda_2) \xi_{{r}},
\]
and that $\xi_{{r}}$ is the identity in $\End(M^{(\lambda_1, \lambda_2)})$. This will turn out to be related to the identity $[e,f]=h$ in $\mathfrak{sl}_2$ (see Theorem \ref{gl_schur_equivalence}).
\end{example}

\begin{example}
Consider the case $l(\lambda)=1$. We obtain precisely Harman's integral permutation Deligne category $\underline{{\rm Perm}}_{\lambda_1}$ (as in Chapter 4 of \cite{NateThesis}). Harman used certain truncations of this category to prove stability properties for modular representations of symmetric groups.
\end{example}

\begin{proposition} \label{enhanced_tensor_rule}
In the case $l(\lambda)=1$, suppose that we define a tensor product on $\mathcal{C}_{\lambda}$ as follows. On objects,
\[
M^{{\alpha}} \otimes M^{{\beta}} = 
\bigoplus_{{\gamma} \in T_\lambda({\alpha}, {\beta})} M^{{\gamma}},
\]
where $M^{{\gamma}}$ is understood to correspond to the vector whose entries are the entries of ${\gamma}$. The tensor product of morphisms is given by generalising Proposition \ref{tensor_rule} in the same way that Definition \ref{enhanced_composition_rule} generalised Proposition \ref{composition_rule} (i.e. allowing ``diagonal'' entries of $T$ to depend on $\lambda$, and requiring all other entries to be nonnegative integers subject to row and column sum constraints). Then,
\begin{enumerate}
\item The tensor product $\xi_{{r}} \otimes \xi_{{s}}$ is given by a finite sum of $\xi_{{q}}$ with coefficients in $\mathcal{R}_l$.
\item The specialisation functor $F_n : \mathcal{C}_\lambda \to S_n\mbox{-mod}$ is a tensor functor (i.e. $F_n(\xi_{{r}} \otimes \xi_{{s}}) = F_n(\xi_{{r}}) \otimes F_n(\xi_{{s}})$).
\item The tensor product is associative.
\end{enumerate}
\end{proposition}

\begin{proof}
The proof of parts $(2)$ and $(3)$ of this theorem are completely analogous to the proof of Proposition \ref{enhanced_properties}. However, we emphasise that the assumption $l(\lambda) = 1$ is important, otherwise $T_\lambda({\alpha}, {\beta})$ is infinite, and the tensor product of two objects is not a finite sum of objects. In this case, however, it is finite; an element ${q}$ of $T({\alpha}, {\beta})$ has entries bounded by row sums or column sums except for the intersection of the first row and first column. The single remaining entry is determined by the fact that all entries must sum to $\lambda_1$.
\end{proof}
\begin{remark}
In the setting of the Deligne category, the monoidal structure on $\underline{{\rm Rep}}(S_t)$ was already defined on the pre-Deligne category $\underline{{\rm Rep}}_0(S_t)$. However, for $\mathcal{C}_\lambda$ we must pass at least to the additive envelope of $\mathcal{C}_\lambda^{(0)}$, because the tensor product of two $M^{{\alpha}}$ is typically a sum of many such terms.
\end{remark}

\begin{remark}
By way of comparison, the elements of $T_{(\lambda_1, \lambda_2)}((\lambda_1,\lambda_2),(\lambda_1,\lambda_2))$ are indexed by $r \in \mathbb{Z}_{\geq 0}$ (so there are infinitely many). In matrix form, they are:
\[
\left( {\begin{array}{cc}
   \lambda_1 - r & r \\
   r & \lambda_2 - r
  \end{array} } \right).
\]
\end{remark}
\noindent
To relate our categories to representations of symmetric groups, we interpolate Lemma \ref{eigenvalue_lemma}. 

\begin{proposition} \label{interpolated_eigenvalue}
Consider the generating function in Lemma \ref{eigenvalue_lemma} (we replace the partition $\lambda$ with $\alpha \in T_\lambda$), and expand the determinant factors as power series ``about the diagonal'', e.g. the second term would be
\[
\det \left( {\begin{array}{cc}
   x_{11} & x_{12} \\
   x_{21} & x_{22}
  \end{array} } \right)^{\alpha_2 - \alpha_3}
  =
  (x_{11}x_{22})^{\alpha_2 - \alpha_3} \sum_{m \geq 0} { \alpha_2 - \alpha_3 \choose m} \left(-\frac{x_{12}x_{21}}{x_{11}x_{22}}\right)^{m}.
\]
Let the coefficient of $x_{ij}^{q_{ij}}$ be $c_q$. Then, $c_q \in \mathcal{R}_l$, and moreover, $\xi_q \mapsto c_q$ makes $\mathcal{R}_l$ into a module over $\End(M^{\alpha})$.
\end{proposition}

\begin{proof}
Note that it makes sense to ask for the coefficient of $x_{ij}^{q_{ij}}$ because in any term of the sum, the exponent of $x_{ij}$ is in $\mathbb{Z}_{\geq 0}$ if $i \neq j$, and $\lambda_i + \mathbb{Z}$ if $i=j$. The $c_q$ are in $\mathcal{R}_l$, because they are sums of products of binomial coefficients of $(\alpha_i - \alpha_{i+1}) \in \lambda_i-\lambda_{i+1} + \mathbb{Z}$. The module statement follows from our usual Zariski-dense specialisation argument (as in Proposition \ref{enhanced_properties}).
\end{proof}

\begin{definition}
Let $HS^{\alpha}$ be the left module over $\End(M^\alpha)$ defined in Proposition \ref{interpolated_eigenvalue}, and let ${}^\alpha HS$ be the analogously defined right module.
\end{definition}

\begin{example} \label{two_row_module}
Consider the action of $\xi_q$ on $HS^{(\lambda_1,\lambda_2)}$, where
\[
q = \left( {\begin{array}{cc}
   \lambda_1 - m & m \\
   m & \lambda_2 - m
  \end{array} } \right).
\]
The construction in Proposition \ref{interpolated_eigenvalue} immediately gives that $\xi_q$ acts by multiplication by $(-1)^m {\lambda_2 \choose m}$.
\end{example}

\begin{theorem}
In the case $l(\lambda)=1$, $\mathcal{C}_{\lambda}$ is a symmetric monoidal category, where the tensor product is given by Proposition \ref{enhanced_tensor_rule}. Moreover, the specialisation functor $F_d : \mathcal{C}_\lambda \to S_d$-mod is a tensor functor.
\end{theorem}

\begin{proof}
Note that $M^{(\lambda_1)}$ is the identity object for the tensor product. The only part of this theorem that is not immediately implied by Proposition \ref{enhanced_tensor_rule} is the symmetry (note that we have an associator that is essentially the same as in Remark \ref{associator_remark}). The symmetric operation is given by ``taking the transpose'' of double cosets. Recall that $M^\alpha \otimes M^\beta$ is a direct sum of certain $M^\gamma$; we have
\begin{eqnarray*}
c: M^{{\alpha}} \otimes M^{{\beta}} &\to& M^{{\beta}} \otimes M^{{\alpha}} \\
c(M^{{\gamma}}) &=& M^{{\gamma}^T},
\end{eqnarray*}
as per Remark \ref{reflection_remark}. This is readily checked to respect associativity, although we omit the details.
\end{proof}

\begin{remark}
If we had not assumed that $l(\lambda)=1$, then the tensor product of two objects would be an infinite sum of objects.
\end{remark}

\begin{theorem} \label{Deligne_equivalence}
If $l(\lambda)=1$ and we work over a field of characteristic zero, $\mathcal{C}_{\lambda_1}$ is equivalent to $\underline{{\rm Rep}}(S_{\lambda_1})$ as a symmetric tensor category.
\end{theorem}

\begin{proof}
We consider a functor $Q:\underline{{\rm Rep}}_0(S_{\lambda_1}) \to \mathcal{C}_{\lambda_1}$, show that it is symmetric monoidal, and that passing to the Karoubian envelope (i.e. the Deligne category) makes $Q$ an equivalence. The functor $Q$ is defined as follows, in analogy to Proposition \ref{tensor_pow_dec}. On objects,
\[
Q([n]) = \bigoplus_{
\substack{
\mbox{set partitions $\alpha$ of $\{1,2,\ldots,n\}$}
}} N^\alpha,
\]
where $N^\alpha$ is isomorphic to $M^{(\lambda_1-l(\alpha), 1^{l(\alpha)})}$. Recall Definition \ref{D_definition}, which gave a way of constructing a partition diagram $D(q)$ from a double coset $q$. We use the analogous extension to elements $q \in T_{\lambda_1}$. On morphisms, $Q(x_{D({q})}) = \xi_{{q}}$. To see that $Q$ is actually a functor, we need to check that it respects composition of morphisms. As in the proof of Proposition \ref{enhanced_properties}, we show this by specialising to symmetric groups $S_d$. We have that $x_{D({r})} x_{D({s})}$ is a sum of $x_{D({q})}$ with coefficients that are polynomials in $\lambda_1$. Similarly, $\xi_{{r}} \xi_{{s}}$ is a linear combination of $\xi_{{q}}$ with coefficients that are polynomials in $\lambda_1$. We need to check that the polynomial coefficients are the same, but we know that $x_{D({q})}$ and $\xi_{{q}}$ agree under the specialisation functors to $S_d$-mod (from the Deligne category and $\mathcal{C}_{\lambda_1}$ respectively). Because these functors do not annihilate $x_{D({q})}$ or $\xi_{{q}}$ provided $d$ is sufficiently large, it follows that the polynomial coefficients agree when evaluated at $n \gg 0$, and hence are equal as polynomials by a Zariski-density argument.
\newline \newline \noindent
An analogous argument shows that $Q$ is a tensor functor (by comparing $x_{D({r})} \otimes x_{D({s})}$ and $\xi_{{r}} \otimes \xi_{{s}}$), and in fact a symmetric tensor functor (because the specialisation functors both respect the symmetric structure). Thus, $Q$ is an embedding of symmetric monoidal categories. It remains to show that when we pass to the Karoubian envelope (and recover $\underline{{\rm Rep}}(S_{\lambda_1})$), then $Q$ becomes an equivalence. For this, it suffices to check that any $M^{{\alpha}}$ is a summand of an object of the form $M^{(\lambda_1 - m, 1^m)}$. For this, we take $m = \alpha_2 + \alpha_3 + \cdots + \alpha_{l(\alpha)}$ (and note that $\lambda_1 - m = \alpha_1$). For a permutation $\sigma$ in $S_m$, let $q(\sigma)$ be the double coset represented by the block sum of the $1\times 1$ matrix $[\lambda_1-m]$ with permutation matrix of $\sigma$. We check that $\xi_{q(\sigma)} \xi_{q(\sigma^\prime)} = \xi_{q(\sigma^\prime\sigma)}$. When considering matrices $A = (a_{ijk})$ as in Definition \ref{enhanced_composition_rule}, the only non-zero entries other than $a_{1,1,1}=\lambda_1-r$ must be of the form $a_{i+1, \sigma(i)+1, \sigma^{\prime}(\sigma(i))+1}$ and take the value $1$ (moreover $\sigma(i)$ determines both $i$ and $\sigma^\prime(\sigma(i))$, so the combinatorial multiplier is always $1$).
\newline \newline \noindent
Now define $S = S_{\alpha_2}\times S_{\alpha_3} \times \cdots \times S_{\alpha_{l(\alpha)}}$ and
\[
e_\alpha = \frac{1}{|S|}\sum_{\sigma \in S} \xi_{q(\sigma)}.
\]
It immediately follows that $e_\alpha$ is an idempotent (it is an idempotent in $kS^{op}$, and the multiplication here is the same). To check that this idempotent defines a summand of $M^{(\lambda_1 - m, 1^m)}$ isomorphic to $M^\alpha$, we again specialise to finite symmetric groups, and simply note that the specialisation of $e_\alpha$ turns the permutation representation on cosets of $S_{(n-m, 1^m)}$ into the permutation representation on cosets of $S_{(n-m, \alpha_2, \ldots, \alpha_{l(\alpha)})}$. Now the required properties of morphisms in our categories easily follow by further Zariski-density arguments.

\end{proof}
\noindent
In fact, this isomorphism only holds in characteristic zero. In positive characteristic, the functor $Q$ still exists, but there are fewer idempotents in the endomorphism algebras of the Deligne category, and not every $M^{\alpha}$ will arise upon taking the Karoubian envelope.

\begin{proposition}
Suppose that $l(\lambda)=1$, and that the ground ring is a field of characteristic $p > 0$. Then $\underline{{\rm Rep}}(S_{\lambda_1})$ embeds into $\mathcal{C}_\lambda$, but it is not monoidally equivalent.
\end{proposition}

\begin{proof}
The functor $Q$ from the proof of Theorem \ref{Deligne_equivalence} is an embedding (the proof is the same as in that case). Suppose that there was a monoidal equivalence; as $\underline{\rm Rep}(S_t)$ is tensor-generated by the object $[1]$, the functor is determined on objects by the image of $[1]$. The image of $[1]$ must be a direct summand of some $\bigoplus_{\tau} M^{(\lambda_1 - |\tau|, \tau)}$. This means that the image of $[1]^{\otimes r}$ must be a summand of 
\[
(\bigoplus_{\tau} M^{(\lambda_1 - |\tau|, \tau)})^{\otimes r}.
\]
Consider the tensor-product rule in Definition \ref{enhanced_tensor_rule}; $M^\alpha \otimes M^\beta$ is a sum of $M^\gamma$ for $\gamma \in T_\lambda(\alpha, \beta)$. Such $\gamma$ can be thought of as matrices with row and column sums $\beta$ and $\alpha$ respectively. Note that unless $i=j=1$, one of $\alpha_j$ and $\beta_i$ is an element of $\mathbb{Z}_{\geq 0}$ (rather than $\lambda_1 + \mathbb{Z}$), so in this case $\gamma_{ij}$ is bounded by some entry of $\alpha$ or $\beta$. This means that
\[
(\bigoplus_{\tau} M^{(\lambda_1 - |\tau|, \tau)})^{\otimes r}
\]
is a direct sum of $M^{(\lambda_1 - |\rho|, \rho)}$ for compositions $\rho$, where the parts $\rho_i$ are bounded by the largest part $\tau_j$ of any $\tau$. We let $p^a$ be a power of the prime $p$ that is larger than this bound.
\newline \newline \noindent
Now it suffices to show that $M^{(\lambda_1 - p^a, p^a)}$ cannot arise as a direct summand of an object of the form $\bigoplus M^{(\lambda_1 - |\rho|, \rho)}$. If it was a direct summand, we would have maps
\[
M^{(\lambda_1-p^a,p^a)} \to \bigoplus_\rho M^{(\lambda_1 - |\rho|, \rho)} \to M^{(\lambda_1-p^a,p^a)}
\]
composing to the identity. We have a 1-dimensional module $HS^{(\lambda_1-p^a,p^a)}$ over $\End(M^{(\lambda_1-p^a,p^a)})$, so in particular, for some $\rho$, we must have maps
\[
M^{(\lambda_1-p^a,p^a)} \to M^{(\lambda_1 - |\rho|, \rho)} \to M^{(\lambda_1-p^a,p^a)}
\]
whose composition acts nontrivially on the $\End(M^{(\lambda_1-p^a,p^a)})$-module $HS^{(\lambda_1-p^a,p^a)}$. We show this cannot happen.
\newline \newline \noindent
Note that maps $\xi_r$ from $M^{(\lambda_1-|\rho|,\rho)}$ to $M^{(\lambda_1-p^a,p^a)}$ are indexed by matrices of the form
\[
r = \left( \begin{array}{cccc}
\lambda_1-|\rho|-\epsilon_0 & \rho_1-\epsilon_1 & \cdots & \rho_m-\epsilon_m \\
\epsilon_0 & \epsilon_1 & \cdots & \epsilon_m
\end{array} \right),
\]
where $0 \leq \epsilon_i \leq \rho_i < p^a$ for $i \geq 1$, while $\epsilon_0 = p^a - \sum_{i=1}^m \epsilon_i \geq 0$. Maps $\xi_s$ in the opposite direction are indexed by transposes of such matrices, but we will write $\delta_i$ in place of $\epsilon_i$ and denote such a matrix $s$. Now, we consider $\xi_r \xi_s \in \End(M^{\lambda_1-p^a,p^a})$ and the resulting action on $HS^{(\lambda_1-p^a,p^a)}$. Example \ref{two_row_module} calculates that the scalar by which $\xi_q$ acts on $HS^{(\lambda_1-p^a,p^a)}$, where
\[
q = \left( {\begin{array}{cc}
   \lambda_1 - b & b \\
   b & p^a - b
  \end{array} } \right).
\]
This scalar, $(-1)^b {p^a \choose b}$, vanishes modulo $p$ unless $b=0$ or $b=p^a$, so we will only need to consider the $\xi_q$ of this form.
\newline \newline \noindent
Let $A_{ijk}$ be a 3-tensor as in Definition \ref{enhanced_composition_rule}; it has dimensions $2 \times (m+1) \times 2$. Note that $A_{i 1 k}$ is of the form (where we label rows and columns by their respective sums)
\begin{center}
\begin{tabular}{ cc|c| } 
 \multicolumn{1}{l|}{$a_{i1k}$} & $\lambda_1-|\rho|-\delta_0$ & $\delta_0$ \\  \hline
 \multicolumn{1}{l|}{$\lambda_1-|\rho|-\epsilon_0$} & $a_{111}$ & $a_{112}$ \\  \hline
 \multicolumn{1}{l|}{$\epsilon_0$} & $a_{211}$ & $a_{212}$ \\ 
 \hline
\end{tabular},
\end{center}
while for $j>1$, $A_{ijk}$ has the form
\begin{center}
\begin{tabular}{ cc|c| } 
 \multicolumn{1}{l|}{$a_{ijk}$} & $\rho_{j-1}-\delta_{j-1}$ & $\delta_{j-1}$ \\  \hline
 \multicolumn{1}{l|}{$\rho_{j-1}-\epsilon_{j-1}$} & $a_{1j1}$ & $a_{1j2}$ \\  \hline
 \multicolumn{1}{l|}{$\epsilon_{j-1}$} & $a_{2j1}$ & $a_{2j2}$ \\ 
 \hline
\end{tabular}.
\end{center}
The only way that we may obtain one of the two $\xi_q$ that act nontrivially on $HS^{(\lambda_1-p^a,p^a)}$ is if either $a_{2j1}=a_{1j2}=0$ for all $j$, or $a_{2j2}=0$ for all $j$ (this is because we need $b=0$ or $b=p^a$). In the first case, comparing row and column sums shows $\epsilon_j = \delta_j$ for all $j$ (including $j=0$).
\newline \newline \noindent
In the case $a_{2j1}=a_{1j2}=0$, we obtain
\[
\sum_{\{\epsilon_j\}} {\lambda_1 - p^a \choose \{\lambda_1 - |\rho| - \epsilon_0 \} \cup \{\rho_j-\epsilon_j\}_{j \geq 1} }{p^a \choose \{\epsilon_j\}} \xi_q,
\]
where the sum is over all possible choices of $\{\epsilon_i \}$, and each factor is a multinomial coefficient. Now, the second binomial coefficient vanishes modulo $p$ unless $\epsilon_j = 0$ for $j \geq 1$ (our earlier bound, $\epsilon_j < p^a$ for $j \geq 1$, rules out the possibility that $\epsilon_j=p^a$). Hence, $\epsilon_0=p^a$. Thus only one summand contributes, and we obtain
\[
{\lambda_1 - p^a \choose (\lambda_1 - |\rho| - p^a, \rho )}\xi_q,
\]
where the scalar is a multinomial coefficient. Note that in this case $\xi_q$ acts by the scalar $(-1)^0{p^a \choose 0} = 1$.
\newline \newline \noindent
In the case $a_{2j2}=0$, we obtain 
\[
\sum_{\{\epsilon_j\}} {\lambda_1 - 2p^a \choose \{\lambda_1 - |\rho|-\delta_0-\epsilon_0 \} \cup \{\rho_j-\delta_j-\epsilon_j\}_{j \geq 1} }{p^a \choose \{\delta_j\}}{p^a \choose \{\epsilon_j\}}\xi_q.
\]
As before, the only way to obtain something nonzero modulo $p$ is for $\epsilon_j = 0$ for $j \geq 1$, so $\epsilon_0=p^a$, and similarly for $\delta_j$. We get
\[
{\lambda_1 - 2p^a \choose \{\lambda_1 - |\rho|-2p^a \} \cup \{ \rho_i \}}\xi_q.
\]
Note that in this case $\xi_q$ acts by the scalar $(-1)^{p^a} {p^a \choose p^a} = (-1)^{p^a}$.
\newline \newline \noindent
Because $(-1)^{p^a}$ is always congruent to $-1$ modulo $p$, to conclude that the total action is zero it suffices to see that, modulo $p$,
\[
{\lambda_1 - p^a \choose \{\lambda_1-|\rho|-p^a \}\cup \{ \rho_i \}} = {\lambda_1 - 2p^a \choose \{\lambda_1-|\rho|-2p^a \} \cup \{ \rho_i \}},
\]
which follows from the following identity. Let us abbreviate the binomial coefficients to (for example) ${\lambda_1-p^a  \choose \rho}$. Then,
\[
{(\lambda_1-2p^a)+p^a  \choose \rho}
= \sum_{\rho^\prime}
{\lambda_1-2p^a  \choose \rho-\rho^\prime}
{p^a \choose \rho^\prime},
\]
where the sum is over compositions $\rho^\prime$, and $\rho-\rho^\prime$ indicates componentwise subtraction. Note that a term in the sum is zero unless the parts of $\rho^\prime$ are bounded by those of $\rho$, hence by $p^a$. Then, ${p^a \choose \rho^\prime}$ is divisible by $p$ unless $\rho^\prime$ is the empty composition. This gives the required congruence.

\end{proof}

\begin{proposition}
The category $\mathcal{C}_\lambda$ is a module category over $\mathcal{C}_{|\lambda|}$ (where $|\lambda| = \lambda_1 + \cdots + \lambda_l$), with the tensor product given by the formula in Proposition \ref{enhanced_tensor_rule}. If $\mu \models d$, then the specialisation functor $F_d \otimes F_\mu : \mathcal{C}_{|\lambda|} \otimes \mathcal{C}_{\lambda} \to S_d\mbox{-mod} \otimes S_d\mbox{-mod}$ is compatible with the usual tensor structure of $S_d$-mod.
\end{proposition}

\begin{proof}
The proof is completely analogous to the proof of Proposition \ref{enhanced_properties}. The only point that should be made is that a tensor product of two objects gives a finite sum of objects. This is because if $\alpha \in T_{|\lambda|}$ and $\beta \in T_\lambda$, we need to consider $\gamma$ whose row and column sums are $\alpha$ and $\beta$ respectively. However, all entries of $\alpha$ except $\alpha_1$ are nonnegative integers, and hence give a bound on the entries of $\gamma$ that are not in the first row. The first row is determined by the other entries using the column-sum conditions.
\end{proof}

\noindent
We now discuss tensor ideals in $\mathcal{C}_\lambda$, which allow us to define a quotient module category. In characteristic zero, it turns out that if $\lambda_i \notin \mathbb{Z}_{\geq 0}$, there will not be any proper tensor ideals, making $\mathcal{C}_\lambda$ into a ``simple" module category.
\begin{definition}
Suppose that $\mathcal{C}$ is a tensor category, and $\mathcal{M}$ is a module category (so there is an action map $\mathcal{C} \times \mathcal{M} \to \mathcal{M}$). A tensor ideal $I$ of $\mathcal{M}$ is a subspace $I(M_1, M_2)$ of each hom-space $\Hom_{\mathcal{M}}(M_1, M_2)$ for every $M_1, M_2 \in \mathcal{M}$ such that the following properties hold.
\begin{enumerate}
\item The composition of a morphism in $I$ with a morphism in $\mathcal{M}$ (on either the left or right) is another morphism in $I$.
\item The tensor product of a morphism in $I$ with a morphism in $\mathcal{C}$ (on either the left or right) is another morphism in $I$.
\end{enumerate}
In this situation, one may define the quotient category $\mathcal{M}/I$ with the same objects as $\mathcal{M}$, but whose hom-sets are
\[
\Hom_{\mathcal{M}/I} (M_1, M_2) = \Hom_{\mathcal{M}}(M_1, M_2)/I(M_1, M_2).
\]
This is a well-defined category, which inherits the structure of a module category over the tensor category $\mathcal{C}$.
\end{definition}

\begin{proposition}
Suppose that we work over a field of characteristic zero, and $\lambda$ has been specialised to take values in the ground field, i.e. each $\lambda_i$ is a scalar rather than a variable. Let $J$ be a subset of $\{1,2, \ldots, l(\lambda)\}$. Define $I_J(M^\alpha, M^\beta)$ to be the span of $\xi_q$ ($q \in T_\lambda(\alpha, \beta)$) such that $q_{ii} \in \mathbb{Z}_{<0}$ for some $i \in J$. Then $I_J$ is a tensor ideal.
\end{proposition}
\begin{proof}
Consider $\xi_r \xi_q$ with the assumption that $\xi_q \in I_J$, and in particular $q_{mm} \in \mathbb{Z}_{<0}$ with $m \in J$. By using Definition \ref{enhanced_composition_rule}, let $A_{ijk}$ be a 3-tensor such that $\sum_{i} A_{ijk} = q_{jk}$, but $\sum_{j} A_{mjm} \geq 0$. Then, the coefficient arising for such an $A$ is
\[
\prod_{i,k} \frac{\left(\sum_j A_{ijk}\right)!}{\prod_j A_{ijk}!}.
\]
The term of this product corresponding to $i=k=m$ is a fraction whose numerator is the factorial of a nonnegative integer, but whose denominator involves the factorial of a negative integer (formally equal to infinity). As a result, the quantity is zero, so $\xi_r \xi_q \in I_J$. The case of $\xi_q \xi_r$ is identical. The argument for $\xi_r \otimes \xi_q$ is similar. Using Proposition \ref{enhanced_tensor_rule}, and considering tensors $T_{ijkl}$ arising from summands $\xi_T$ in $\xi_r \otimes \xi_q$, the constraint $\sum_{ik} T_{imkm} = q_{mm} < 0$ shows that $T_{ijkl}$ must have a negative entry (necessarily at $(i,j,k,l)=(m,m,m,m)$) if $q$ does. \end{proof}
\noindent
Of course, if $J \subseteq J^\prime$, then $I_J \subseteq I_{J^\prime}$. Hence among these tensor ideals there is a largest one, corresponding to $J = \{1,2,\ldots, l(\lambda)\}$. However, these need not be distinct, for instance, if all $\lambda_i$ are not integers, then any $I_J$ is zero. One reason to consider $I_J$ is if some $\lambda_i$ is a negative integer; in this case, if $i \in J$ then $I_{J}$ contains every morphism in the category. Analogously, if $\lambda_i \notin\mathbb{Z}$, then adding or removing $i$ from $J$ does not change the tensor ideal $I_J$. Thus the largest proper tensor ideal among these corresponds to $J_{max} = \{1 \leq i \leq l(\lambda) \mid \lambda_i \in \mathbb{Z}_{\geq 0}\}$. In view of the similarity to the specialisation functors $F_\mu$ (whose kernels are spanned by $\xi_q$ for $q$ that have a negative integer entry when $\lambda$ is evaluated at $\mu$), quotienting by $I_J$ may be viewed as a ``partial specialisation''. We now show that the quotient by the largest proper tensor ideal among the $I_J$ yields a simple module category (i.e. one with no nontrivial tensor ideals).
\begin{theorem} \label{tensor_ideal_theorem}
Suppose that we work over a field of characteristic zero. Let $J_{max} = \{1 \leq i \leq l(\lambda) \mid \lambda_i \in \mathbb{Z}_{\geq 0}\}$, and consider $\mathcal{C}_\lambda/I_{J_{max}}$. This module category has no nontrivial tensor ideals.
\end{theorem}
\begin{proof}
It suffices to show that any tensor ideal of $\mathcal{C}_\lambda$ properly containing $I_{J_{max}}$ is all of $\mathcal{C}_\lambda$. Suppose that $I$ is such a tensor ideal. Then $I$ contains a morphism $M^\alpha \to M^\beta$ which is a linear combination of some $\xi_q$, at least one of which has no diagonal entries that are negative integers. Among such $q$, let $q^{*}$ be the one with maximal sum of off-diagonal entries. Let us tensor with $M^{(|\lambda|-1,1)}$ (an object of $\mathcal{C}_{|\lambda|}$) and consider the tensor product of the identity of $M^{(|\lambda|-1,1)}$ with $\xi_{q}$. By the tensor product rule, we obtain a sum of $\xi_T$, where $T_{ijkl}$ satisfies
\begin{eqnarray*}
\sum_{ik} T_{ijkl} &=& q_{jl} \\
\sum_{jl} T_{1j1l} &=& |\lambda|-1 \\
\sum_{jl} T_{1j2l} &=& 0 \\
\sum_{jl} T_{2j1l} &=& 0 \\
\sum_{jl} T_{2j2l} &=& 1.
\end{eqnarray*}
So, $T$ is essentially obtained from $q$ by decrementing an entry, say $q_{ij}$, and adding a new row and column with a 1 at their intersection (considered a diagonal entry). This is a morphism from $M^{\alpha^\prime}$ to $M^{\beta^\prime}$, where $\alpha^\prime$ is obtained from $\alpha$ by decrementing $\alpha_j$ and adding a part of size $1$ and similarly $\beta^\prime$ is obtained from $\beta$ by decrementing $\beta_i$ and adding a part of size $1$. We may then compose on the left by projector onto $M^{\beta^\prime}$ as a summand of $M^{(|\lambda|-1,1)} \otimes M^{\beta}$ and on the right by the projector onto $M^{\alpha^\prime}$ as a summand of $M^{(|\lambda|-1,1)} \otimes M^\alpha$. This has the effect of decrementing the $(i,j)$-th entry of $q$ in all $\xi_q$, and appending a new row and column with a one at their intersection (considered as diagonal, so the number of off-diagonal entries decreases). Decrementing an off-diagonal entry equal to zero yields a zero morphism. The upshot of this is that we may iterate this procedure, decrementing an arbitrary sequence of coordinates $(i,j)$ until $\xi_{q^*}$ becomes the identity morphism (all off-diagonal entries equal to zero). Because of our assumption that $q^*$ had the largest sum of off-diagonal entries among $\xi_q$ occuring in our original morphism, it follows that all other $\xi_q$ are killed by this process. Hence our tensor ideal must contain the identity morphism of some $M^{\alpha}$. In fact, we may continue this operation to obtain identity morphisms for more and more $M^\alpha$; we obtain any $\alpha$ of the form 
\[
\alpha = (\lambda_1 - \rho_1, \ldots, \lambda_l - \rho_l, 1^{\sum_i \rho_i}).
\]
where $\lambda_i - \rho_i$ are less than or equal to the diagonal entries of $q^*$ (in particular, we may take $\alpha$ with entries that are not negative integers). Note that so far we have not used our assumption about the values of $\lambda_i$, nor about the characteristic of the ground field.
\newline \newline \noindent
Now we show that the identity morphism of any object factors through one of the $M^\alpha$ that we have constructed. Let 
\[
\beta = (\lambda_1 - \sigma_1, \ldots, \lambda_l - \sigma_l, \tau_1, \ldots, \tau_r)
\]
and let us take $\alpha$ with $\rho_k \geq \sigma_k$. Then we may consider $\xi_r$, where for $1 \leq i,j \leq l$, $r_{ij} = \delta_{ij}(\lambda_i - \rho_i)$, and all other entries are zero or one such that $r \in T_\lambda(\alpha, \beta)$. Let us compute $\xi_{r^T} \xi_r$. We consider $A_{ijk}$ with sums $\sum_k A_{ijk} = (r^T)_{ij} = r_{ji}$ and $\sum_i A_{ijk} = r_{jk}$. For any fixed $j$, there is a unique value $m_j$ such that $r_{jm_j} \neq 0$. Thus $A_{ijk} = 0$ unless $i=k=m_j$. But this means that $\sum_j A_{ijk}$ is a diagonal matrix. So the result of this computation is a scalar multiple of the identity morphism. It remains to calculate the scalar. This amounts to keeping track of the values of $A_{m j m}$ for each $m$. They are either $\lambda_k-\sigma_k$ (which happens once for each of the first $l$ diagonals), or $1$ (and the multiplicity is determined by the total sum). Hence, we obtain
\begin{eqnarray*}
& &\prod_{k = 1}^l \frac{(\lambda_k - \sigma_k)!}{(\lambda_k - \rho_k)! 1!^{\rho_k - \sigma_k}} \times \prod_{i > l} \frac{\tau_i!}{1!^{\tau_i}} \\
&=& \prod_{k=1}^l {\lambda_k - \sigma_k \choose \rho_k - \sigma_k}(\rho_k - \sigma_k)! \times \prod_{i>l} \tau_i!.
\end{eqnarray*}
The only way this expression could vanish in characteristic zero is if one of the binomial coefficients was zero. But this can only happen if some $\lambda_k$ is an integer and $0 \leq \lambda_k - \sigma_k < \rho_k-\sigma_k$. In particular, this would imply $\lambda_k < \rho_k$. However, we could assume that $\lambda_k - \rho_k$ are nonnegative integers. This completes the proof.
\end{proof}
\noindent
In particular, $\mathcal{C}_\lambda$ can be thought of as being ``generically simple'' as a module category. It is also minimal in a certain sense, as the following proposition demonstates.

\begin{proposition}
Let $\mathcal{D}$ be a Karoubian subcategory of $\mathcal{C}_\lambda$ such that $\mathcal{D}$ contains $M^\alpha$ for all $\alpha \in T_\lambda$ and $\mathcal{D}$ is a module subcategory of $\mathcal{C}_\lambda$. Then $\mathcal{D} = \mathcal{C}_\lambda$.
\end{proposition}
\begin{proof}
What we must prove is that $\mathcal{D}$ contains all morphisms of $\mathcal{C}_\lambda$, in particular by Theorem \ref{gl_schur_equivalence} (proved in the next section), the morphisms in the category are generated by $\xi_q$ where $q$ has a single nonzero off-diagonal entry. Thus it suffices to show that all these morphisms are contained in $\mathcal{D}$.
\newline \newline \noindent
Choose $q \in T_\lambda(\alpha, \beta)$ with the single nonzero off-diagonal entry $q_{ij} = n$. (In particular, this means that $\alpha_k = \beta_k$ except that $\alpha_i = \beta_i -n $ and $\alpha_j = \beta_j + n$.) We use the module subcategory assumption to tensor the identity morphism of $M^\alpha$ by the (unique up to scalar) morphism $f: M^{(|\lambda|)} \to M^{(|\lambda|-n, n)}$ in $\mathcal{C}_{|\lambda|}$. Using Proposition \ref{enhanced_tensor_rule} to compute the tensor product, we find $M^{(|\lambda|)} \otimes M^{\alpha} = M^\alpha$, and $M^{(|\lambda|-n,n)} \otimes M^{\alpha}$ is a direct sum of $M^\gamma$, where the $\gamma$ are obtained from $\alpha$ by choosing nonnegative integers $\alpha_i^\prime$ summing to $n$, so that $\gamma = (\alpha_i-\alpha_i^\prime, \alpha_i^\prime) \in T_\lambda$. The component of $f \otimes Id_{M^\alpha}$ mapping to $M^\gamma$ is $\xi_r$, where $r$ is obtained from the matrix of $Id_{M^\alpha}$ by subtracting $\alpha_i^\prime$ from the $(i,i)$-entry, and adding a new row corresponding to $\alpha_i^\prime$ containing the entry $\alpha_i^\prime$ in column $i$ (otherwise zero). In particular, by choosing $\gamma$ such that all $\alpha_i^\prime$ are zero except for one; we obtain a morphism whose diagram has a single off-diagonal entry (which equals $n$). We are not yet done because the row of this entry is not arbitrary.
\newline \newline \noindent
Now we perform the ``transpose'' construction, tensoring the identity morphism of $M^\beta$ by the unique morphism $g: M^{(|\lambda|-n, n)} \to M^{(|\lambda|)}$. We obtain a summand $M^{\gamma^\prime}$ of $M^{(|\lambda|-n,n)} \otimes M^\beta$; $\gamma^\prime$

are obtained from $\beta$ by choosing nonnegative integers $\beta_i^\prime$ summing to $n$, so that $\gamma^\prime = (\beta_i-\beta_i^\prime, \beta_i^\prime) \in T_\lambda$. The component of $g \otimes Id_{M^\beta}$ mapping to $M^{\gamma^\prime}$ is $\xi_r$, where $s$ is obtained from the matrix of $Id_{M^\beta}$ by subtracting $\beta_i^\prime$ from the $(i,i)$-entry, and adding a new column corresponding to $\beta_i^\prime$ containing the entry $\beta_i^\prime$ in column $i$ (otherwise zero). We may project out the summand for which $\gamma=\gamma^\prime$, so we may form $\xi_s \xi_r : M^\alpha \to M^\beta$ (a morphism in $\mathcal{D}$), and observe that $\xi_s \xi_r = \xi_q$. Thus $\mathcal{D}$ contains a generating set of morphisms in $\mathcal{C}_\lambda$ and must therefore be all of $\mathcal{C}_\lambda$.

\end{proof}

\section{Interpretation using Lusztig's enveloping algebra $\dot{U}(\mathfrak{gl}_\infty)$}
\noindent
In this section, we explain how the categories $\mathcal{C}_\lambda$ may be described using a modified version of the universal enveloping algebra of $\mathfrak{gl}_\infty$.
\newline \newline \noindent
The following theorem of Doty and Giaquinto \cite{DG} sets the stage for our construction.
\begin{theorem}
Let $\Lambda(n,d)$ be the set of compositions of $n$ of length at most $d$, which we think of as a subset of the weight lattice of $\mathfrak{gl}_n$. The Schur algebra $S_{\mathbb{Q}}(n,d)$ is the associative algebra over $\mathbb{Q}$ generated by $1_\lambda$ for $\lambda \in \Lambda(n,d)$ and $e_i, f_i$ for $1 \leq i \leq n-1$, subject to the following relations.
\begin{eqnarray*}
1_\lambda 1_\mu &=& \delta_{\lambda, \mu}1_\lambda \\
\sum_{\lambda \in \Lambda(n,d)} 1_\lambda &=& 1 \\
e_i 1_\lambda &=& 
\begin{cases} 
      1_{\lambda+\alpha_i} e_i & \mbox{if }\lambda + \alpha_i \in \Lambda(n,d) \\
      0 & \mbox{otherwise} 
\end{cases} \\
f_i 1_\lambda &=& 
\begin{cases} 
      1_{\lambda-\alpha_i} f_i & \mbox{if }\lambda - \alpha_i \in \Lambda(n,d) \\
      0 & \mbox{otherwise} 
\end{cases} \\
1_\lambda e_i &=& 
\begin{cases} 
      e_i 1_{\lambda-\alpha_i} & \mbox{if }\lambda - \alpha_i \in \Lambda(n,d) \\
      0 & \mbox{otherwise} 
\end{cases} \\
1_\lambda f_i &=& 
\begin{cases} 
      f_i 1_{\lambda+\alpha_i} & \mbox{if }\lambda + \alpha_i \in \Lambda(n,d) \\
      0 & \mbox{otherwise} 
\end{cases} \\
{}[e_i, [e_i, e_{i \pm 1}]] &=& 0 \\
{}[f_i, [f_i, f_{i \pm 1}]] &=& 0 \\
{}[e_i,f_j] &=& \delta_{i,j}\sum_{\lambda \in \Lambda(n,d)}(\lambda_i - \lambda_{i+1}) 1_\lambda
\end{eqnarray*}
Here, $\alpha_i = e_i-e_{i+1}$ is the $i$-th simple root of $\mathfrak{gl}_n$. The integral Schur algebra (i.e. over $\mathbb{Z}$) is the subalgebra generated by the idempotents $1_\lambda$ and the divided powers $e_i^r / r!, f_i^r/r!$ ($r \in \mathbb{Z}_{\geq 0}$).
\newline \newline \noindent
Additionally, in the natural action of $S(n,d)$ on $(\mathbb{Q}^{\oplus d})^{\otimes n}$, $1_\lambda$ is projection onto the $\lambda$-weight space (which can be identified with $M^\lambda$).
\end{theorem}
\noindent
We construct an interpolated version of the above presentation that is suited to our purpose.
\begin{definition}
Consider the inclusion $\Lambda(n,d) \hookrightarrow \Lambda(n,d+1)$ by appending a part of size zero to a composition. In the above construction, this leads to a (non-unital) map $S(n,d) \to S(n,d+1)$, and we may define the limit of the resulting directed system:
\[
S(n, \infty) = \varinjlim S(n,d),
\]
noting that the result is a non-unital algebra.
\end{definition}
\noindent
This inherits an action on $E =(\mathbb{Q}^{\oplus \infty})^{\otimes n}$ because the maps respect the corresponding inclusions $(\mathbb{Q}^{\oplus d})^{\otimes n} \hookrightarrow (\mathbb{Q}^{\oplus (d+1)})^{\otimes n}$.
For any fixed composition $\lambda$, we define $1_\lambda$ via the obvious identifications for varying $d$. Then $1_\lambda E = M^\lambda$ as a representation of $S_n$. These statements also hold integrally.

\begin{lemma} \label{action_lemma}
The action of $e_k^m/m!$ on $M^\lambda \subseteq V^{\otimes r}$ is the same as $\xi_s$, where $s = (s_{ij})$ has $s_{ii} = \lambda_i$ for $i \neq k+1$, and all other entries equal to zero except for $s_{k,k+1}=m$, $s_{k+1,k+1} = \lambda_{k+1}-m$. (As usual, if some entry is negative, we take the associated map to be zero.) Similarly, the action of $f_k^m/m!$ is the same as $\xi_{q}$ where $q = (q_{ij})$ has $q_{ii} = \lambda_i$ for $i \neq k$, and all other entries equal to zero except for $q_{k+1,k} = m$ and $q_{kk} = \lambda_k - m$.
\end{lemma}
\begin{proof}
The action of $e_k$ on $V$ is to replace the standard basis vector $v_{k+1}$ with the standard basis vector $v_k$. Because $e_k$ is primitive (in $U(\mathfrak{gl}_n)$, from which the Schur algebra action on $V^{\otimes r}$ is inherited), the action of the divided power $e_k^m/m!$ on a pure tensor $v_{j_1} \otimes v_{j_2} \otimes \cdots \otimes v_{j_r} \in V^{\otimes n}$ is to give the sum of all pure tensors $v_{j_1^\prime} \otimes v_{j_2^\prime} \otimes \cdots \otimes v_{j_r^\prime}$ where exactly $m$ indices that are equal to $k+1$ have been replaced by $k$ (and the others left unchanged); the factor of $m!$ accounds for the possible orderings of the affected indices. But that is exactly the action in Proposition \ref{schur_basis_theorem}. The proof for $f_k^m/m!$ is identical.
\end{proof}

\noindent
Using this result we can give an alternative description of the category $\mathcal{C}_\lambda$ analogously to the Serre presentation of $\mathfrak{gl}_n$. First we construct an interpolated version of the previous algebra presentation.

\begin{definition}
Let $\lambda = (\lambda_1, \lambda_2, \ldots, \lambda_l)$ be variables, and let us work over $\mathcal{R}_l \otimes_{\mathbb{Z}} \mathbb{Q} = \mathbb{Q}[\lambda_1, \lambda_2, \ldots, \lambda_l]$. We define $\mathfrak{U}_\lambda$ to be the quiver algebra with relations as follows. The vertices of the quiver are indexed by $T_\lambda$. For each ${\beta} \in T_\lambda$, there is an edge labelled $e_i$ from ${\beta}$ to ${\beta} + \alpha_i$ (defining the zero map unless ${\beta} + \alpha_i \in T_\lambda$) and an edge $f_i$ from ${\beta}$ to ${\beta} - \alpha_i$ (defining the zero map unless ${\beta} - \alpha_i \in T_\lambda$). We write $1_{{\beta}}$ for the idempotent associated to a vertex ${\beta}$, and the relations are
\begin{eqnarray*}
(e_if_j - f_je_i)1_{{\beta}} &=& \delta_{ij} ({\beta}_i - {\beta}_{i+1}) 1_{{\beta}} \\
(e_i^2 e_{i\pm 1} - 2 e_i e_{i\pm 1}e_i + e_{i \pm 1} e_i^2) 1_{{\beta}} &=& 0 \\
(f_i^2 f_{i\pm 1} - 2 f_i f_{i\pm 1}f_i + f_{i \pm 1} f_i^2) 1_{{\beta}} &=& 0,
\end{eqnarray*}
for any ${\beta} \in T_\lambda$. We also define $\mathfrak{U}_\lambda^\mathbb{Z}$ to be the $\mathcal{R}_l$-subalgebra of $\mathfrak{U}_\lambda$ generated by the divided power elements $e_i^{r}1_{{\beta}}/r!$ and $f_i^r 1_{{\beta}}/r!$.

\end{definition}
\noindent
Note that $\mathfrak{U}_\lambda$ does not contain elements $e_i$, $f_j$, or a multiplicative identity (these would lie in a certain completion). As long as at least one idempotent $1_\beta$ is contained in a monomial, we obtain a well defined element of $\mathfrak{U}_\lambda$.

\begin{remark}
The above definition is similar to Lusztig's quantum group $\dot{U}(\mathfrak{gl}_n)$ at $q=1$, except that it replaces $n$ with $\infty$, and has the variables $\lambda_i$ (which would be zero in the case of $\dot{U}(\mathfrak{gl}_n)$), as well as nonnegativity conditions on certain weights.
\end{remark}

%

%

\begin{definition}
Let $\lambda = (\lambda_1, \lambda_2, \ldots, \lambda_l)$ be a collection of variables. Let $\mathcal{C}_{\lambda}^{\mathfrak{gl},(0)}$ be the category with objects $M^{{\alpha}}$ indexed by ${\alpha} \in T_\lambda$. Define
\[
\Hom(M^{{\alpha}}, M^{{\beta}}) = 1_{{\beta}} \mathfrak{U}_\lambda^{\mathbb{Z}} 1_{{\alpha}}
\]
and define composition of morphisms by mutiplication in $\mathfrak{U}_\lambda^{\mathbb{Z}}$. Let $\mathcal{C}_\lambda^{\mathfrak{gl}}$ be the Karoubian envelope of this category.
\end{definition}

\begin{proposition}
Fix a composition $\mu$ of $n$. There is a specialisation functor $F_\mu:\mathcal{C}_\lambda \to S_n$-mod, which evaluates $\lambda$ at $\mu$.
\end{proposition}

\begin{proof}
Let $\phi_\mu$ be the map which evaluates $\lambda$ at $\mu$. We take $F_\mu(M^\alpha) = M^{\phi_\mu(\alpha)}$ (where, as usual, if $\phi_\mu(\alpha)$ has a negative entry, then $M^{\phi_\mu(\alpha)} = 0$), and on morphisms, $F_\mu$ is defined via the map $\mathfrak{U}_\lambda^{\mathbb{Z}} \to S(n, \infty)$ defined by mapping each generator to the one denoted by the same symbol. This is well defined because $S(n,\infty)$ is obtained from $\mathfrak{U}_\lambda^{\mathbb{Z}}$ by quotienting out all $1_\beta$ for which $\phi_\mu(\beta)$ contains a negative entry.
\end{proof}


\begin{theorem} \label{gl_schur_equivalence}
There is an equivalence of categories $F: \mathcal{C}_\lambda^{\mathfrak{gl}} \to \mathcal{C}_\lambda$. This equivalence respects the specialisation functors $F_\mu$ defined on each category.
\end{theorem}

\begin{proof}
We work over $\mathbb{Q}$, and observe that our construction applies to the integral versions of these categories. We define the functor $F$ on the ``scaffold'' versions of each category, i.e. $\mathcal{C}_\lambda^{(0)}$ and $\mathcal{C}_\lambda^{\mathfrak{gl},(0)}$, and then pass to the Karoubian envelope. We let $F(M^{{\beta}}) = M^{{\beta}}$, and let $F(e_i 1_{{\beta}}) = \xi_{{s}}$ where ${s}$ is the matrix $s$ given in Lemma \ref{action_lemma}. We similarly define $F(f_i1_{{\beta}}) = \xi_{{q}}$ (again from Lemma \ref{action_lemma}) and $F(1_{{\beta}}) = \xi_{{r}}$, where $r_{ij} = \delta_{ij}\beta_i$. Because the algebra $\mathfrak{U}_\lambda$ is generated by these elements, the category $\mathcal{C}_\lambda^{\mathfrak{gl},(0)}$ is generated by these morphisms. To check that this actually defines a functor, we must show that $F$ respects the relations between the generators.
\newline \newline \noindent
First we check the Chevalley relations (commutation relation between $e_i$ and $f_j$). Suppose that $m \neq n$, and let us calculate $e_m f_n 1_{{\beta}}$. We consider matrices $A = (a_{ijk})$ as per Definition \ref{enhanced_composition_rule}. The only off-diagonal entry of ${s}$ is ${s}_{m,m+1} = 1$, and the only off-diagonal entry of ${q}$ is ${q}_{n+1,n} = 1$. By the property that $\sum_k a_{ijk} = s_{ij}$, for $i \neq m$, $a_{ijk}$ can only be nonzero when $i=j$. Similarly for $k \neq n$, $a_{ijk}$ can only be nonzero when $j=k$. So except for these two cases $a_{ijk} = \delta_{ij}\delta_{jk}{\beta}_j$. This leaves us to determine the cases where $i=m$ (and $j=m$ or $j=m+1$), as well as those where $k=n$ (and $j=n$ or $j=n+1$). This amounts to $(i,j,k) \in \{ (m,m,m), (m,m+1,m+1), (n+1,n+1,n),(n,n,n)\}$ (all other remaining cases automatically being zero because of the condition that $m \neq n$). But the sum conditions now give $a_{m,m+1,m+1} = 1$, $a_{m,m+1,m} = {\beta}_m-1$, as well as $a_{n+1,n+1,n} = 1$, $a_{n,n+1,n} = {\beta}_n-1$. Thus we obtain a single $\xi_{{\gamma}}$, where $\gamma_{ij} = \delta_{ij}{\beta}_i - \delta_{im}\delta_{jm}-\delta_{in}\delta_{jn} + \delta_{im}\delta_{j,m+1} + \delta_{i,n+1}\delta_{jn}$. A similar calculation shows that $f_ne_m1_{{\beta}}$ gives the same result. The case $m=n$ is Example \ref{chevalley_reln}.\
\newline \newline \noindent
It remains to check the Serre relations; we only check $(e_i^2 e_{i+1} - 2e_i e_{i+1} e_i + e_{i+1}e_i^2)1_{{\beta}} = 0$ for $i=1$, the other cases being analogous. This reduces verifying the identity to a calculation involving the first $3$ parts of ${\beta}$, so we assume that ${\beta} = (\beta_1, \beta_2, \beta_3)$ has only $3$ parts. Let
\[
p^{(1)} = \left( \begin{array}{ccc}
\beta_1 & 0 & 1 \\
0 & \beta_2 & 0 \\
0 & 0 & \beta_3 - 1
\end{array} \right), \hspace{10mm}
p^{(2)} = \left( \begin{array}{ccc}
\beta_1 & 1 & 0 \\
0 & \beta_2 - 1 & 1 \\
0 & 0 & \beta_3 - 1
\end{array} \right), \hspace{10mm}
p^{(3)} = \left( \begin{array}{ccc}
\beta_1 & 2 & 0 \\
0 & \beta_2 - 2 & 0 \\
0 & 0 & \beta_3
\end{array} \right)
\]
Then, applying Definition \ref{enhanced_composition_rule} gives 
\begin{eqnarray*}
F(e_{i+1}1_{\beta + \alpha_i}) F(e_{i}1_\beta) &=& 
\xi_{p^{(2)}} \\
F(e_{i}1_{\beta + \alpha_{i+1}}) F(e_{i+1}1_\beta) &=&
\xi_{p^{(1)}} + \xi_{p^{(2)}} \\
F(e_{i}1_{\beta + \alpha_i}) F(e_{i}1_\beta) &=& 2 \xi_{p^{(3)}}.
\end{eqnarray*}
Now we complete the calculation. Let
\[
r = \left( \begin{array}{ccc}
\beta_1 & 1 & 1 \\
0 & \beta_2 - 1 & 0 \\
0 & 0 & \beta_3 - 1
\end{array} \right), \hspace{10mm}
s = \left( \begin{array}{ccc}
\beta_1 & 2 & 0 \\
0 & \beta_2 - 2 & 1 \\
0 & 0 & \beta_3 - 1
\end{array} \right).
\]
\noindent
Then,
\begin{eqnarray*}
F(e_{i}1_{\beta + \alpha_i + \alpha_{i+1}})
\xi_{p^{(1)}} = \xi_{r} \\
F(e_{i}1_{\beta + \alpha_i + \alpha_{i+1}})
\xi_{p^{(2)}} = \xi_{r} + 2 \xi_{s}\\
F(e_{i+1}1_{\beta + 2\alpha_i})
\xi_{p^{(3)}} = \xi_{s}.
\end{eqnarray*}
\noindent
Combining these equations gives $F(e_i^2 e_{i+1}1_{{\beta}}) = 2 \xi_{{r}} + 2\xi_{{s}}$ and $F(e_{i+1}e_i^2 1_{{\beta}}) = 2 \xi_{{s}}$, while $F(e_ie_{i+1}e_i 1_{{\beta}}) =  \xi_{{r}} + 2 \xi_{{s}}$. Thus the Serre relations follow.
\newline \newline \noindent
To show $F$ is an equivalence, we must show it is fully faithful. We prove the ``full'' part by showing that the $F(e_i1_\beta)$ and $F(f_j1_\beta)$ generate $\mathcal{C}_\lambda^{(0)}$. We prove the faithful part by a dimension count.
\newline \newline \noindent
Firstly, let us consider the elements $E_{ij}$ defined by $E_{ij} = [\ldots[[e_i, e_{i+1}],e_{i+2}], \ldots, e_{j-1}]$ for $i<j$, and $E_{ij} = [\ldots[[f_{i-1}, f_{i-2}],f_{i-3}], \ldots, f_{j}]$ for $i>j$; these are the elementary matrices in the usual definition of $\mathfrak{gl}_n$. Then, if we define $E_{kk}1_\beta = \beta_k$, we have $[E_{ij},E_{kl}]1_\beta = \delta_{jk}E_{il}1_\beta-\delta_{li}E_{kj}1_\beta$. It then follows from Lemma  \ref{action_lemma} that $F(E_{ij}1_\beta) = \xi_q$, where $q$ has a single off-diagonal entry at coordinates $(i,j)$, or equivalently, $q_{ij} = \beta_i \delta_{ij} - \delta_{jj} + \delta_{ij}$. We show that any $\xi_q$ lies in the subalgebra generated by the $F(E_{ij}1_\beta)$. Let $S$ be the set of $q$ for which this is the case, and note that we obtain any $\xi_q$ where $q$ has a single off-diagonal entry by Lemma \ref{action_lemma}, by using a divided power of a suitable $E_{ij}$. We now show that every upper-triangular $q$ may be obtained by taking a product of such elements.
\newline \newline \noindent
Consider 
\[
\prod_{i<j} \frac{E_{ij}^{q_{ij}}}{q_{ij}!}1_\beta,
\]
where the terms in the product are ordered so that the value of $j$ decreases from left to right, and the order among terms with equal $j$ is arbitrary. One checks that the result is $\xi_p$, where $p$ has entries below the diagonal that are equal to zero, and $p_{ij} = q_{ij}$ for $i<j$ (i.e. above the diagonal). One way to see this is to remember that the action of $E_{ij}$ is to replace vectors $v_j$ with $v_i$ in a tensor product $V^{\otimes d}$. When considering the product above, if an $E_{ij}$ occurs to the left of $E_{jk}$, then the composite action might turn a vector $v_k$ into $v_j$ and then $v_i$. With our choice of ordering of factors, this cannot happen. This means the number of $v_j$ turned into $v_i$ is exactly $q_{ij}$. However, that is exactly the action of $\xi_q$.
\newline \newline \noindent
Thus $S$ contains all upper-triangular $q$, and similarly all lower-triangular $q$. It now remains to check that this implies all $q$ are in $S$. Let $q^{(+)}$ be the element constructed above whose entries agree with $q$ above the diagonal (and are zero below), and $q^{(-)}$ be the analogous element for entries below diagonal. Now we observe that $\xi_{q^{(-)}} \xi_{q^{(+)}} = \xi_q + \ldots$, where the omitted terms have a smaller sum of off-diagonal entries than $q$ (this follows from the discussion of the filtration $\mathcal{F}_n^\prime$ in the next paragraph because all binomial coefficients in the associated graded computation are equal to $1$). Thus by induction on the sum of off-diagonal entries of $q$, $S$ consists of all possible $q$, and hence $F$ is full, and this holds for the integral version of the category.
\newline \newline \noindent
By the PBW theorem, elements of hom-spaces of $\mathcal{C}_\lambda^{\mathfrak{gl},(0)}$ are a linear combination of terms of the form
\[
1_\alpha \prod_{i \neq j} E_{ij}^{q_{ij}} 1_\beta,
\]
where the choice of ordering in the product is not important (we may even choose different orderings for different elements). We claim that if there is an $i_0$ such that $\sum_{j \neq i_0} q_{i_0 j} > \beta_{i_0}$, there is a choice of ordering of terms in the product that makes the monomial zero (and so it may be omitted from a spanning set). We choose all monomials $E_{i_0 j}^{q_{i_0 j}}$ to appear at the right of the product. Now we note that the product of this subset of monomials,
\[
1_{\alpha^\prime} \prod_{j \neq i_0}E_{i_0 j}^{q_{i_0 j}} 1_\beta,
\]
must be zero, because calculating weights gives $\alpha_{i_0}^\prime = \beta_{i_0} - \sum_{j \neq i_0} q_{i_0 j} < 0$, hence $\alpha \notin T_\lambda$. To deduce the fact that $F$ is faithful from this, we consider the following filtrations.
\newline \newline \noindent
Let $\mathcal{F}_n$ ($n \in \mathbb{Z}_{\geq 0}$) be the PBW filtration on $1_{\beta}\mathfrak{U}1_\alpha$ (each $E_{ij}$ for $i \neq j$ lies in filtration degree 1). Let $\mathcal{F}_n^\prime$ be the filtration on hom-spaces of $\mathcal{C}_\lambda^{(0)}$ induced by declaring that $\xi_q$ should have filtration degree $\sum_{i \neq j}q_{ij}$. Let us check that this is indeed a filtration. Consider the product of $\xi_r$ and $\xi_s$; let $A_{ijk}$ be such that $\sum_{i} A_{ijk} = s_{jk}$ and $\sum_k A_{ijk} = r_{ij}$. We obtain a sum of $\xi_q$, where $q_{ik} = \sum_j A_{ijk}$. The filtration degree is maximised when $A_{ijk}$ has the largest sum of entries with $i \neq k$ as possible. This in turn is achieved by taking $A_{iik} = s_{ik}$ and $A_{ikk} = r_{ik}$, $A_{jjj} = \lambda_j - \sum_{i \neq j}r_{ij} - \sum_{k\neq j} s_{jk}$ and all other entries zero. This is because for fixed $j$, the sum of off-diagonal entries of $A_{ijk}$ is bounded by $\sum_{i} r_{ij} + \sum_{k} s_{jk}$; the first term covers all rows except the $j$-th, and the second covers all columns except the $j$-th (together this covers all entries but the $j$-th diagonal). Our construction attains this bound (in fact, it is unique, because the bound double-counts entries not in the $j$-th row or column, so any configuration with a nonzero entry outside of these would be strictly less than this bound). This also shows that in the associated graded algebra, 
\[
\xi_r \xi_s = \prod_{i \neq k} \frac{(r_{ik}+s_{ik})!}{r_{ik}!s_{ik}!} \xi_q,
\]
where $q_{ik} = r_{ik} + s_{ik}$ for $i \neq k$.
\newline \newline \noindent
The functor $F$ sends $\mathcal{F}_n$ to $\mathcal{F}_n^\prime$ by construction (the $n$-th divided power of $e_i$ or $f_j$ corresponds to a single off-diagonal entry equal to $n$). However, we know that $\mathcal{F}_n^\prime$ is spanned by $q$ which index a basis of $\mathcal{F}_n$. Thus the dimension of the former spaces is less than or equal to that of the latter space, but the map is a surjection upon taking the associated graded spaces. Hence $F$ must be injective.
\newline \newline \noindent
The fact that $F$ respects the specialisation functors follows immediately from Lemma \ref{action_lemma}. The equivalence $F: \mathcal{C}_\lambda^{\mathfrak{gl}} \to \mathcal{C}_\lambda$ holds integrally (i.e. over $\mathcal{R}_l$). This is because the proof constructed all $\xi_q$ using divided powers of $E_{ij}$.


\end{proof}

\begin{proposition} \label{lie_alg_action}
Suppose $l(\lambda)=1$. The tensor structure on $\mathcal{C}_{\lambda}^{\mathfrak{gl}}$ obtained from $\mathcal{C}_{\lambda}$ via $F$ is as follows. Define $E_{ij} = [\ldots[[e_i, e_{i+1}],e_{i+2}], \ldots, e_{j-1}]$ for $i<j$, $E_{ij} = [\ldots[[f_{i-1}, f_{i-2}],f_{i-3}], \ldots, f_{j}]$ for $i>j$. On the level of objects,
\[
M^{{\alpha}} \otimes M^{{\beta}} = \bigoplus_{{\gamma} \in T_\lambda({\alpha}, {\beta})} M^{{\gamma}}.
\]
For morphisms, 
\[
E_{i,j}1_{{\alpha}} \otimes 1_{{\beta}} = \bigoplus_{{\gamma} \in T_\lambda(\alpha, \beta)} \sum_{k} E_{(i,k),(j,k)}1_{{\gamma}},
\]
where the index $(i,k)$ (as well as $(j,k)$) indicates the coordinate corresponding to $\gamma_{ij}$ (because $\gamma_{ij}$ is viewed as a vector rather than a matrix). An analogous formula holds for the second tensor factor.
\end{proposition}

\begin{proof}
It suffices to notice that this formula holds under the specialisation functors to $S_d$, where $E_{ij}$ may be interpreted as acting on $V^{\otimes d}$ by replacing a pure tensor with the sum of pure tensors obtained by replacing a tensor factor $v_j$ with $v_i$.
\end{proof}
\noindent
We conclude this section by explicitly comparing the two realisations of $\mathcal{C}_\lambda$ in the case where $l(\lambda)=2$.
\begin{proposition} \label{chevalley_to_schur}
Let $l(\lambda) = 2$, and let 
\[
q(m,n) = \left( \begin{array}{ccc}
\lambda_1 -n & m \\
n & \lambda_2 - m
\end{array} \right).
\]
Then we have the equality of generating functions
\[
\sum_{m,n} x^m y^n \xi_{q(m,n)} = \exp(yf) (1-xy)^{h_2} \exp(xe)1_\lambda,
\]
where $e, f$ are Chevalley generators and $h_2$ satisfies $h_2 1_\alpha = \alpha_2 1_\alpha$.
\end{proposition}
\begin{proof}
By Lemma \ref{action_lemma}, $e^m/m! \cdot 1_\lambda$ is equal to $\xi_{q(m,0)}$ (and $f^n/n! \cdot 1_\lambda = \xi_{q(0,n)}$). This allows us to write
\[
\exp(yf) (1-xy)^{h_2} \exp(xe) 1_\lambda
=
\sum_{m = 0}^\infty \sum_{n = 0}^\infty 
y^n \xi_{r(m, n)}
(1-xy)^{\lambda_2 -m}x^m \xi_{q(m,0)},
\]
where
\[
r(m,n) = \left( \begin{array}{ccc}
\lambda_1 -n+m & 0 \\
n & \lambda_2 -m
\end{array} \right).
\]
One can readily check that
\[
\xi_{r(m,n)} \xi_{q(m,0)}
= \sum_{i=0}^{\min(m,n)} {\lambda_2-m+i \choose i } \xi_{q(m-i,n-i)}.
\]
Thus our expression becomes
\[
\sum_{m = 0}^\infty \sum_{n = 0}^\infty 
y^n (1-xy)^{\lambda_2 -m}x^m \sum_{i=0}^{\min(m,n)} {\lambda_2-m+i \choose i } \xi_{q(m-i,n-i)}.
\]
Changing variables to $m^\prime = m-i$ and $n^\prime = n-i$ gives
\[
\sum_{i=0}^\infty
\sum_{m^\prime = 0}^\infty \sum_{n^\prime = 0}^\infty 
y^{n^\prime +i} (1-xy)^{\lambda_2 -m^\prime -i}x^{m^\prime + i}  {\lambda_2-m^\prime \choose i } \xi_{q(m^\prime,n^\prime)}.
\]
Now we observe that the sum over $i$ may be simplified using the binomial theorem:
\[
\sum_{m^\prime = 0}^\infty \sum_{n^\prime = 0}^\infty 
\left(1 + \frac{xy}{1-xy} \right)^{\lambda_2 - m^\prime}
y^{n^\prime } (1-xy)^{\lambda_2 -m^\prime}x^{m^\prime} \xi_{q(m^\prime,n^\prime)} = \sum_{m^\prime, n^\prime} x^{m^\prime}y^{n^\prime} \xi_{q(m^\prime, n^\prime)}.
\]
\end{proof}
\begin{remark}
In Proposition \ref{chevalley_to_schur}, let us substitute $e, f, h_2$ to be the elementary matrices $E_{12}, E_{21}, E_{22}$, so that the right hand side of the equation becomes a $2\times 2$ matrix. (Note that $(1-xy)^{E_{22}} = \exp(\log(1-xy)E_{22})$.) We obtain a matrix $g$:
\[
g = \left( \begin{array}{ccc}
1 & 0 \\
y & 1
\end{array} \right)
\left( \begin{array}{ccc}
1 & 0 \\
0 & 1-xy
\end{array} \right)
\left( \begin{array}{ccc}
1 & x \\
0 & 1
\end{array} \right)
=
\left( \begin{array}{ccc}
1 & x \\
y & 1
\end{array} \right).
\]
Recall from Section 2.4 in \cite{Green} (used in the proof of Lemma \ref{eigenvalue_lemma}) that an element $g = (g_{ij})_{ij}$ of $GL_2$ acts on $(\mathbb{C}^{\oplus 2})^{\otimes d}$ by 
\[
g \mapsto  \sum_{q} \xi_q \prod_{ij} g_{ij}^{q_{ij}},
\]
which is precisely the left hand side of the equation we started with.
\end{remark}

\section{Categorified Stability of Kronecker Coefficients}
\noindent
We conclude by proving stability properties of symmetric group representations.
\begin{proposition}
When viewed as a module over $\End(M^{{\alpha}})$ in $\mathcal{C}_\lambda^{\mathfrak{gl}}$, the action of $\mathfrak{U}_\lambda$ on $HS^\alpha$ is as a highest weight space. That is, any monomial 
\[
\prod_{r\in R} f_{i_r} \prod_{s \in S} e_{j_s} 1_{{\alpha}}
\]
acts as zero unless $S$ is empty (and hence $R$ is also empty by weight considerations).
\end{proposition}

\begin{proof}
This follows from the fact that $HS^{{\alpha}}$ was constructed as an interpolation of actions on highest-weight spaces.
\end{proof}

\begin{corollary}
We may use the realisation in terms of $\mathcal{C}_\lambda^{\mathfrak{gl}}$ to describe the left $\End(M^\beta)$-module
\[
\Hom(M^{{\alpha}}, M^{{\beta}}) \otimes_{\End(M^{{\alpha}})} HS^{{\alpha}} \cong 1_{{\beta}}\mathfrak{U}_{\lambda}1_{{\alpha}}
\otimes_{1_{{\alpha}}\mathfrak{U}_{\lambda}1_{{\alpha}}} HS^{{\alpha}}.
\]
In particular, this space is finite dimensional.
\end{corollary}

\begin{proof}
The only thing that needs to be checked is that the space is finite dimensional. But this follows from the highest-weight property of $HS^{{\alpha}}$, because the space is spanned by elements whose first tensor factor (in $\Hom(M^{{\alpha}}, M^{{\beta}})$) does not contain any $e_i$, hence only $f_j$'s. The weight condition immediately determines the number of each $f_j$, so only their ordering is not fixed. Hence there are finitely many possibilities.
\end{proof}
\noindent
The next proposition explains that objects of the form $\Hom(M^{{\alpha}}, M^{{\beta}}) \otimes_{\End(M^{{\alpha}})} HS^{{\alpha}}$ (which has the structure of a free $\mathcal{R}_l$-module) interpolate $\Hom_{S_n}(S^{\lambda}, M^{{\alpha}})$.

\begin{proposition}
Applying specialisation functors $F_\mu: \mathcal{C}_\lambda \to S_d$-mod gives a map 
\[
F_\mu: \Hom(M^{{\alpha}}, M^{{\beta}}) \otimes_{\End(M^{{\alpha}})} HS^{{\alpha}} \to \Hom(M^{\phi_\mu({\alpha})}, M^{\phi_\mu({\beta})}) \otimes_{\End(M^{\phi_\mu({\alpha})})} HS^{\phi_\mu({\alpha})},
\]
which yields the module $\Hom_{S_{|\mu|}}(S^{\phi_\mu({\alpha})}, M^{\phi_\mu({\beta})})$.
\end{proposition}
\begin{proof}
This follows from Lemma \ref{sss_lemma}.
\end{proof}

\begin{theorem}
Let $X$ be an object of the Deligne category. Then, consider
\[
{}^{\beta}HS \otimes_{\End(M^\beta)}
\Hom(M^{{\alpha}}, X \otimes M^{{\beta}}) \otimes_{\End(M^{{\alpha}})} HS^{{\alpha}}.
\]
This is a finite dimensional space which interpolates the spaces $\Hom_{S_d}(S^{\phi_\mu({\alpha})}, F(X) \otimes S^{\phi_\mu({\beta})})$ functorially in $X$. 
\end{theorem}
\begin{proof}
This follows similarly from Lemma \ref{sss_lemma}.
\end{proof}
\noindent
Consider the case where $X = X_\nu$ is an indecomposable object of the Deligne category indexed by the partition $\nu$ (so that the specialisation functor to $S_{|\mu|}-$mod maps $X$ to $S^{(|\mu|-|\nu|,\nu)}$ provided $|\mu|$ is sufficiently large). Then, for any $\mu$ a partition of $d$, we may apply specialisation functors
\[
{}^{\beta}HS \otimes_{\End(M^\beta)}
\Hom(M^{{\alpha}}, X \otimes M^{{\beta}}) \otimes_{\End(M^{{\alpha}})} HS^{{\alpha}}
\to
\Hom_{S_d}(S^{\phi_\mu(\alpha)}, S^{(|\mu|-|\nu|,\nu)} \otimes S^{\phi_\mu(\beta)}).
\]
We know this map is surjective. By definition, $\alpha, \beta \in T_\lambda$ differ from $\lambda$ by a fixed (finite) vector, hence the same is true of $\phi_\mu(\alpha), \phi_\mu(\beta)$ with respect to $\mu$. Now, the specialisation functors depend polynomially on $\lambda$, and we may choose to specialise $\lambda = m \mu$ (where $\mu$ is a fixed integer partition), to reduce the polynomial dependence to one variable $m$. But then, for generic $m$ (in particular for $m$ sufficiently large), the dimension of the space is constant. Hence, this provides a categorification of the $(|\lambda|, \lambda, \lambda)$ stability patters of Kronecker coefficients in the sense of Stembridge \cite{Stembridge} (we have interpolated the $((|\mu|-|\nu|,\nu), \mu, \mu)$ multiplicity spaces).
In conclusion, the category $\mathcal{C}_\lambda$ categorifies the $(|\lambda|, \lambda, \lambda)$ stability patterns of Kronecker coefficients in the sense of Stembridge.

\appendix
\section{Failure of $\mathcal{C}_\lambda$ to be Krull-Schmidt}
\noindent
It turns out that $\mathcal{C}_\lambda$ is not in general a Krull-Schmidt category when $l(\lambda) > 1$. The following example illustrates this.
\newline \newline \noindent
Let us work over $\mathbb{C}$, and suppose that $l(\lambda)=2$, and let us consider $M^{(\lambda_1, \lambda_2, 1)}$, inside $\mathcal{C}_{(\lambda_1+1, \lambda_2)}$, where $\lambda_1, \lambda_2$ are generic scalars. To show that this object violates the Krull-Schmidt property, we proceed in several steps. For the sake of brevity, we only sketch the argument.
\newline \newline \noindent
{\bf Step 1: The algebra} $A = \End(M^{(\lambda_1, \lambda_2, 1)})$.
\newline \noindent
By the definition of the category $\mathfrak{U}_{(\lambda_1+1, \lambda_2)}$, the endomorphism algebra $A$ can is spanned by paths in a certain quiver algebra with relations. The relations include those of the Serre presentation of $\dot{U}(\mathfrak{gl}_n)$ (where $n$ can take any value), so expressing each path as a product of matrix units $E_{ij}$ (which form a basis of any $\mathfrak{gl}_n$) we may choose a PBW monomial ordering where monomials with $i<j$ appear on the right, and monomials with $i>j$ appear on the left (we do not include terms with $i=j$, because they only contribute a scalar multiple). By weight considerations (i.e. membership of $T_{(\lambda_1+1,\lambda_2)})$, the presence of a single factor $E_{i,j}$ with $i<j$ and $j>3$ yields the zero map in $\End(M^{(\lambda_1, \lambda_2, 1)})$. It follows that the only endomorphisms that need to be considered are those for which $i,j \leq 3$, so that the endomorphism algebra is a quotient of the weight-zero subalgebra of $\dot{U}(\mathfrak{gl}_3)$. Let us write a PBW monomial in this algebra as
\[
E_{3,2}^{b_{3,2}} E_{3,1}^{b_{3,1}} E_{2,1}^{b_{2,1}} E_{1,2}^{a_{1,2}} E_{1,3}^{a_{1,3}} E_{2,3}^{a_{2,3}}.
\]
If $a_{1,3}+a_{2,3} > 1$, then the same weight considerations mean we obtain the zero endomorphism of $M^{(\lambda_1, \lambda_2, 1)}$. It follows that $A$ is the quotient of the weight-zero subalgebra of $\dot{U}(\mathfrak{gl}_3)$ by the ideal $I$ consisting of all PBW monomials with $a_{1,3}+a_{2,3} > 1$. Thus $A$ has a basis consisting of PBW monomials such that $b_{3,2}+b_{3,1}=a_{2,3}+a_{1,3}$, $b_{3,1}+b_{2,1}=a_{1,3}+a_{1,2}$ (weight-zero conditions), and $a_{1,3}+a_{2,3} \leq 1$. To leading order in the PBW filtration, multiplication of PBW monomials is addition of the exponents. Classifying solutions to the above system shows $A$ is generated by the following five elements:
\begin{eqnarray*}
& &E_{2,1}E_{1,2}, \\
& &E_{3,1}E_{1,3}, \\
& &E_{3,2}E_{2,3}, \\
& &E_{3,2}E_{2,1}E_{1,3}, \\
& &E_{3,1}E_{1,2}E_{2,3}.
\end{eqnarray*}
\noindent
Note that inside $A$, $\frac{1}{\lambda_2+1}E_{3,2}E_{2,3}$ is an idempotent, because
\begin{eqnarray*}
(E_{3,2}E_{2,3})^2 1_{(\lambda_1, \lambda_2,1)}&=& E_{3,2}E_{2,3}E_{3,2} 1_{(\lambda_1, \lambda_2+1,0)} E_{2,3} \\
&=& E_{3,2}(E_{3,2}E_{2,3}+\lambda_2+1)1_{(\lambda_1, \lambda_2+1,0)}E_{2,3} \\
&\in& (\lambda_2+1)E_{3,2}E_{2,3}1_{(\lambda_1, \lambda_2,1)} + I.
\end{eqnarray*}
Similarly, $\frac{1}{\lambda_1+1} E_{3,1}E_{1,3}$ is another idempotent in $A$.
\newline \newline \noindent
{\bf Step 2: Understanding} $A$ {\bf via a suitable representation.}
\newline \noindent
The algebra $A$ acts on the $(\lambda_1, \lambda_2, 1)$-weight space of any Verma module of highest-weight dominating $(\lambda_1,\lambda_2,1)$. There are two such families of weights: $(\lambda_1+r, \lambda_2-r, 1)$, and $(\lambda_1+r, \lambda_2-r+1,0)$. We focus only on the latter family. Consider a Verma module of highest weight $(\lambda_1+r, \lambda_2-r+1,0)$ with highest-weight vector $v$. The $(\lambda_1, \lambda_2, 1)$-weight space is two dimensional, with basis
\begin{eqnarray*}
X_1 &=& E_{3,2}\frac{E_{2,1}^r}{r!} v,\\
X_2 &=& E_{3,1}\frac{E_{2,1}^{r-1}}{(r-1)!} v.
\end{eqnarray*}
\noindent
We may calculate the action of the five generating elements of $A$ on this basis, in terms of $r$. For example,
\begin{eqnarray*}
E_{3,2}E_{2,3} X_2 &=& E_{3,2}E_{2,3}E_{3,1}\frac{E_{2,1}^{r-1}}{(r-1)!} v \\
&=& E_{3,2}(E_{3,1}E_{2,3} + E_{2,1})\frac{E_{2,1}^{r-1}}{(r-1)!} v \\
&=& 0 + E_{3,2} E_{2,1}\frac{E_{2,1}^{r-1}}{(r-1)!} v \\
&=& r X_1.
\end{eqnarray*}
The identity 
\[
[e, \frac{f^r}{r!}] = \frac{f^{r-1}}{(r-1)!} (h-r+1)
\]
for $\mathfrak{sl}_2$ is useful for many cases of this calculation. Ultimately we obtain the following representing matrices.
\begin{eqnarray*}
E_{2,1}E_{1,2} &\mapsto&
\left( \begin{array}{cc}
r(\lambda_1-\lambda_2+r) & -r \\
-(\lambda_1-\lambda_2+r) & r(\lambda_1-\lambda_2+r)-\lambda_1+\lambda_2
\end{array} \right)\\
E_{3,1}E_{1,3} &\mapsto& \left( \begin{array}{cc}
\lambda_2 + 1 & r \\
0 & 0
\end{array} \right)\\
E_{3,2}E_{2,3} &\mapsto& \left( \begin{array}{cc}
0 & 0 \\
\lambda_1-\lambda_2+r & \lambda_1+1
\end{array} \right)\\
E_{3,2}E_{2,1}E_{1,3} &\mapsto& \left( \begin{array}{cc}
r(\lambda_1-\lambda_2 + r) & r(\lambda_1+1) \\
0 & 0
\end{array} \right)\\
E_{3,1}E_{1,2}E_{2,3} &\mapsto& \left( \begin{array}{cc}
0 & 0 \\
(\lambda_1-\lambda_2+r)(\lambda_2+1) & r(\lambda_1-\lambda_2+r)
\end{array} \right).
\end{eqnarray*}
\noindent
It is convenient to perform a mild change of basis.
\begin{eqnarray*}
E_{2,1}E_{1,2} + E_{3,1}E_{1,3} + E_{3,2}E_{2,3} - (\lambda_1+2)(\lambda_2+1)
\Id&\mapsto&
(r(\lambda_1-\lambda_2+r)-(\lambda_1+1)(\lambda_2+1))
\left( \begin{array}{cc}
1 & 0 \\
0 & 1
\end{array} \right)\\
E_{3,1}E_{1,3} &\mapsto& \left( \begin{array}{cc}
\lambda_2 + 1 & r \\
0 & 0
\end{array} \right)\\
E_{3,2}E_{2,3} &\mapsto& \left( \begin{array}{cc}
0 & 0 \\
\lambda_1-\lambda_2+r & \lambda_1+1
\end{array} \right)\\
E_{3,2}E_{2,1}E_{1,3}-(\lambda_1+1)E_{3,1}E_{1,3} &\mapsto& \left( \begin{array}{cc}
r(\lambda_1-\lambda_2 + r) -(\lambda_1+1)(\lambda_2+1)& 0 \\
0 & 0
\end{array} \right)\\
E_{3,1}E_{1,2}E_{2,3} -(\lambda_2+1)E_{3,2}E_{2,3}&\mapsto& \left( \begin{array}{cc}
0 & 0 \\
0 & r(\lambda_1-\lambda_2+r)-(\lambda_1+1)(\lambda_2+1)
\end{array} \right).
\end{eqnarray*}
\noindent
As the first matrix is the sum of the last two, we may disregard the first matrix for the purpose of figuring out what algebra is generated by these elements. Let us conjugate by the diagonal matrix with entries $r, 1$, which has the effect of dividing the $(1,2)$ entry by $r$, and multiplying the $(2,1)$ entry by $r$. If we let 
\[
P = r(\lambda_1-\lambda_2+r)-(\lambda_1+1)(\lambda_2+1),
\]
our four generating matrices may now be written
\begin{eqnarray*}
\left( \begin{array}{cc}
\lambda_2 + 1 & 1 \\
0 & 0
\end{array} \right)\\
\left( \begin{array}{cc}
0 & 0 \\
P + (\lambda_1+1)(\lambda_2+1) & \lambda_1+1
\end{array} \right)\\
\left( \begin{array}{cc}
P & 0 \\
0 & 0
\end{array} \right)\\
\left( \begin{array}{cc}
0 & 0 \\
0 & P
\end{array} \right).
\end{eqnarray*}
\noindent
The upshot of this is that instead of considering the action of the algebra $A$ on a single Verma module, one may consider the product of all such modules for $r \in \mathbb{Z}_{>0}$. Considering this much larger representation (and the vectors obtained by taking $X_1$ for each factor, or $X_2$ for each factor), we may now interpret $r$ as a variable, so we have a homomorphism $A \to \Mat_2(\mathbb{C}[r])$. In fact, the $r$-dependence is expressed purely in terms of $P$, so actually we have a homomorphism $A \to \Mat_2(\mathbb{C}[P])$. One can check that the image of this homomorphism is precisely the set of matrices such that at $P=0$, $(-1,\lambda_2+1)^T$ is an eigenvector, while at $P=-(\lambda_1+1)(\lambda_2+1)$, $(1,0)^T$ is an eigenvector. So this algebra consists of $2 \times 2$ polynomial matrices that preserve a flag at one point, and another flag at a different point.
\newline \newline \noindent
{\bf Step 3: Explaining the failure of the Krull-Schmidt property.}
\newline \noindent
Note that our two idempotents are represented by the matrices
\begin{eqnarray*}
\left( \begin{array}{cc}
1 & \frac{1}{\lambda_2+1} \\
0 & 0
\end{array} \right)\\
\left( \begin{array}{cc}
0 & 0 \\
\frac{P + (\lambda_1+1)(\lambda_2+1)}{ \lambda_1+1} & 1
\end{array} \right).
\end{eqnarray*}
\noindent
The first idempotent acts as zero on the invariant line at $P=0$, and as the identity on the invariant line at $P=-(\lambda_1+1)(\lambda_2+1)$. On the other hand, the second idempotent acts as zero on the invariant line at both values of $P$. As elements of $A$ preserve the invariant lines at these values of $P$, conjugating by a unit in $A$ cannot change the scalars by which these elements act on the respective invariant lines. So, there are four conjugacy classes of rank 1 idempotents, corresponding to whether they act by 0 or 1 on each of the two invariant lines. In particular, the two idempotents we considered, together with their complementary idempotents (i.e. $1-e$ for an idempotent $e$) are pairwise nonconjugate. This means that our idempotents (originally defined in $A$) provide inequivalent direct sum decompositions of $M^{(\lambda_1, \lambda_2, 1)}$ in $\mathcal{C}_{(\lambda_1 +1, \lambda_2)}$. In particular, direct sum decompositions are not unique, and hence the Krull-Schmidt property does not hold.

\bibliographystyle{alpha}
\bibliography{ref.bib}

\begin{thebibliography}{BDVO15}

\bibitem[BDVO15]{BDO}
Christopher Bowman, Maud De~Visscher, and Rosa Orellana.
\newblock {The partition algebra and the Kronecker coefficients}.
\newblock {\em Transactions of the American Mathematical Society},
  367(5):3647--3667, 2015.

\bibitem[BHH17]{BHH}
Georgia Benkart, Tom Halverson, and Nate Harman.
\newblock {Dimensions of irreducible modules for partition algebras and tensor
  power multiplicities for symmetric and alternating groups}.
\newblock {\em Journal of Algebraic Combinatorics}, 46(1):77--108, 2017.

\bibitem[CO11]{OC}
Jonathan Comes and Victor Ostrik.
\newblock {On blocks of Deligne's category $\underline{Rep}(S_t)$}.
\newblock {\em Advances in Mathematics}, 226(2):1331 -- 1377, 2011.

\bibitem[DG02]{DG}
Stephen Doty and Anthony Giaquinto.
\newblock {Presenting Schur Algebras}.
\newblock {\em International Mathematics Research Notices},
  2002(36):1907--1944, 2002.

\bibitem[DJ86]{DJ}
Richard Dipper and Gordon James.
\newblock {Representations of Hecke algebras of general linear groups}.
\newblock {\em Proceedings of the London Mathematical Society}, 3(1):20--52,
  1986.

\bibitem[Gre06]{Green}
James~A Green.
\newblock {\em {Polynomial Representations of GL\_n: with an Appendix on
  Schensted Correspondence and Littelmann Paths}}, volume 830.
\newblock Springer, 2006.

\bibitem[Har15]{NateStability}
Nate Harman.
\newblock {Stability and periodicity in the modular representation theory of
  symmetric groups}.
\newblock {\em arXiv preprint arXiv:1509.06414}, 2015.

\bibitem[Har17]{NateThesis}
Nate Harman.
\newblock {\em {Deligne categories and representation stability in positive
  characteristic}}.
\newblock PhD thesis, Massachusetts Institute of Technology, 2017.

\bibitem[SS16]{SS}
Steven~V Sam and Andrew Snowden.
\newblock {Proof of Stembridge’s conjecture on stability of Kronecker
  coefficients}.
\newblock {\em Journal of Algebraic Combinatorics}, 43(1):1--10, 2016.

\bibitem[Ste14]{Stembridge}
John~R Stembridge.
\newblock {Generalized stability of Kronecker coefficients}.
\newblock {\em \emph{Unpublished Manuscript}}, 2014.

\bibitem[SY12]{SY}
Ana~Paula Santana and Ivan Yudin.
\newblock {Characteristic-free resolutions of Weyl and Specht modules}.
\newblock {\em Advances in Mathematics}, 229(4):2578--2601, 2012.

\end{thebibliography}

\end{document}